\documentclass[12pt]{article}
\usepackage{setspace}
\usepackage{latexsym,amsthm,amsmath,amssymb,url}
\usepackage{enumitem}
\usepackage[round]{natbib}
\usepackage{eurosym}

\usepackage{graphics}
\usepackage[usenames]{color}

\usepackage{tikzsymbols}

\theoremstyle{definition}
\newtheorem{lemma}{Lemma}

\newtheorem{corollary}{Corollary}
\theoremstyle{definition}
\newtheorem{definition}{Definition}
\theoremstyle{definition}

\theoremstyle{definition}
\newtheorem*{theorem*}{Theorem}
\theoremstyle{definition}
\newtheorem{theorem}{Theorem}
\newtheorem{proposition}{Proposition}

\theoremstyle{definition}
\newtheorem{assumption}{Assumption}

\newcommand{\abs}[1]{|#1|}
\newcommand{\set}[1]{\left\{#1\right\}}

\newcommand{\kk}[1]{\kappa(#1)}
\newcommand{\kkk}{\kappa}
\newcommand{\aux}{{\rm aux}}
\newcommand{\PDdecomp}{$DP$-Nash subgraph}
\newcommand{\PDdecomps}{$DP$-Nash subgraphs}
\newcommand{\2}{\vspace{0.1cm}}
\newcommand{\induce}[2]{#1 [  #2 ] }

\newcommand{\blue}[1]{{\color{blue}#1}}
\newcommand{\red}[1]{{\color{red}#1}}

\newcommand{\lv}[1]{}

\usepackage[pdftex, citecolor=blue, colorlinks]{hyperref}

\begin{document}

\title{Exact capacitated domination: on the computational complexity of uniqueness}
\author{
Gregory Z. Gutin\footnote{Computer Science Department, Royal Holloway University of London.} \hspace{.1in} Philip R.\ Neary\footnote{Economics Department, Royal Holloway University of London.} \hspace{.1in} Anders Yeo\footnote{IMADA, University of Southern Denmark.} $^{,}$\footnote{Department of Mathematics, University of Johannesburg.}
}

\date{\today}

\maketitle

\begin{abstract}
\noindent
In this paper we consider a local service-requirement assignment problem named {\it exact capacitated domination} from an algorithmic point of view.
This problem aims to find a solution (a Nash equilibrium) to a game-theoretic model of public good provision.
In the problem we are given a {\it capacitated graph}, a graph with a parameter defined on each vertex that is interpreted as the capacity of that vertex.
The objective is to find a {\it $DP$-Nash subgraph}:\ a spanning bipartite subgraph with partite sets $D$ and $P$, called the {\em $D$-set} and {\em $P$-set} respectively, such that no vertex in $P$ is isolated and that each vertex in $D$ is adjacent to a number of vertices equal to its capacity.
We show that whether a capacitated graph has a unique $DP$-Nash subgraph can be decided in polynomial time.
However, we also show that the nearby problem of deciding whether a capacitated graph has a unique $D$-set is {\sf co-NP}-complete.


\end{abstract}


\setstretch{1.3}

\newpage

\section{Introduction}\label{sec:intro}

Given a graph $G = (V, E)$, a {\it dominating set} is a subset of vertices $S$ such that for every vertex $v \in V$, either $v$ belongs to $S$ or $v$ has a neighbour that belongs to $S$.
Dominating sets have found countless applications in the efficient allocation of resources.
One such example is that of communication in wireless ad-hoc networks, in which a peer can only be reached after entering the area that it covers.
In such a setting, if the cost of allocating resources is the same across all locations, then a minimum dominating set describes the least costly way to allocate resources such that every location has network access.

In practice, however, the setting may be more complicated. 
In particular there may be additional strategic, so called game-theoretic, forces at work. 
Given the cost, it is preferable that a neighbouring location supplies instead of vice versa.
Such strategic forces will pressure the requirement of independence on the subset of vertices $S$, since it is never optimal to supply the resource if access is already available.
This is precisely what happens in economic models of public good provision in networks \citep{BramoulleKranton:2007:JET,GaleottiGoyal:2010:RES} where the solutions, the Nash equilibria, correspond to independent dominating sets.

In this paper we consider the {\it exact capacitated domination} problem.
Using our example of communication in wireless ad-hoc networks, there is the additional assumption of a capacity constraint on the number of neighbouring locations that each peer can provide network access to.
If we wish to ensure that every location in the network has access, then it must be the case that any time resources are assigned to a location, these resources must never be less than that location's capacity.
In particular, whenever resources are assigned to a location, that location must specify the neighbouring locations to which it is willing to provide access.
Formally this is described by a {\it $DP$-Nash subgraph} (hereafter a Nash subgraph):\ a spanning bipartite subgraph with partite sets $D$ and $P$, called the {\em $D$-set} and {\em $P$-set} respectively, such that no vertex in $P$ is isolated and each vertex in $D$ is adjacent to a number of vertices equal to its capacity.
A Nash subgraph describes the {\it Nash equilibrium} of the model of public goods on networks with constrained sharing of \cite{GerkeGHN}, while knowledge of only the $D$-set (and hence the $P$-set) describes the {\it Nash equilibrium outcome}.

Our focus in this paper is the issue of uniqueness. 
The rationale for a detailed focus on uniqueness is that economic models possessing a unique equilibrium are as rare as they are useful.
Uniqueness is rare due to the mathematical structure of economic models (formally, the best-response map of \cite{Nash:1951:AM,Nash:1950:PNASUSA} rarely admits a single fixed point).
Uniqueness is useful as (i) it saves the analyst from an ``equilibrium selection'' headache - justifying why one equilibrium is more likely to emerge than another, and (ii) allows those who study \emph{game-design} to be confident in generating a particular outcome (since only one outcome is stable).
It is for this reason that models with unique equilibria are so highly coveted (see for example the model of currency attacks in \cite{MorrisShin:1998:AER}).
From a purely technical perspective the uniqueness issue is interesting since, unlike independent dominating sets, it is possible for a capacitated graph to have a unique Nash subgraph.
It is also possible for a capacitated graph to have more than one Nash subgraph and yet for all of these Nash subgraphs to have the same $D$-set.

Precisely we will focus on the following two problems:

\2

\2

{\sc Nash Subgraph Uniqueness}: {\em decide whether a capacitated graph has a unique Nash subgraph}, and 

\2

\2

{\sc $D$-set Uniqueness}:  {\em decide whether a capacitated graph has a unique $D$-set}.

\2

\2

While the two questions are clearly related, we show that their time complexities are not unless {\sf P}={\sf co-NP}.
That is, we show that {\sc Nash Subgraph Uniqueness} is polynomial-time solvable whereas {\sc $D$-set Uniqueness} is {\sf co-NP}-complete.
In fact, for {\sc $D$-set Uniqueness} we prove the following complexity dichotomy when the capacity of every vertex is equal to the same non-negative integer $k$. 
If $k\ge 2$ then {\sc $D$-set Uniqueness} is {\sf co-NP}-complete whereas if $k\in \{0,1\}$ then {\sc $D$-set Uniqueness} is in {\sf P}.
We note that the proof of Theorem \ref{all_PD} in Appendix A implies that constructing a $DP$-Nash subgraph and, thus, a $D$-set in every capacitated graph is polynomial-time solvable.

The exact capacitation domination problem is similar to to the {\sc Capacitated Domination} problem \citep{CyganPW11,KaoCL15}, where the number of neighbours that any vertex can dominate must not exceed its capacity but a vertex need not dominate up to its capacity.
Clearly the two problems are related as every exact capacitated dominating set is capacitated dominating set, though the reverse is not the case.
Our focus is on exact capacitated dominating sets since these are better-suited to economic environments; the reason is that any subset of vertices that is a capacitated dominating set but not an exact capacitated dominating set is Pareto-inefficient, an economic term that in this setting can be interpreted as meaning ``excessively wasteful''.
(See the discussion in \cite{GerkeGHN}.)


Preliminaries are given in Section \ref{sec:prel}.
To obtain the above polynomial-time complexity results for {\sc Nash Subgraph Uniqueness} and {\sc $D$-set Uniqueness}, we first prove in Section \ref{sec:char} a charaterization of capacitated graphs with unique $D$-sets, which we believe is of interest in its own right.
In Section \ref{sec:DPs}, we show that  {\sc Nash Subgraph Uniqueness} is in {\sf P}.
In Section \ref{sec:Ds} we prove the above-mentioned complexity dichotomy for {\sc $D$-set Uniqueness}.
We conclude the paper in Section \ref{sec:conc}.

\section{Preliminaries}\label{sec:prel}

In this paper, all graphs are undirected, finite, without loops or parallel edges.
Let $G=(V(G),E(G))$ be a graph and $\kappa:\ V\to \mathbb{Z}_{\ge 0}$ a function.
For every $v\in V$ we will call $\kappa(v)$ the {\em capacity} of $v$ and we will call the pair $(G,\kappa)$ a {\em capacitated graph}.
For any subgraph $H$ of $G$, including $G$ itself, we write $d_H(v)$ for the degree of $v \in V(H)$.

We have the following definition.

\begin{definition}\label{def:DPNash}
A spanning subgraph $H$ of $(G, \kappa)$ is called a {\em $DP$-Nash subgraph}, or simply a {\it Nash subgraph}, if $H$ is bipartite with partite sets $D$ and $P$, called the {\em $D$-set} and {\em $P$-set} of $H$ respectively, such that no vertex of $P$ is isolated and for every $x\in D,$ $d_H(x)=\min\{d_G(x),\kappa(x)\}$.
\end{definition}


Since every Nash subgraph $H$ is a bipartite graph, we will write it as the triple $(D,P;E')$, where $D,P$ are $D$-set and $P$-set of $H$, respectively, and $E'$ is the edge set of $H$.
A vertex subset $B\subset V$ is a {\em $D$-set} of capacitated graph $(G, \kappa)$ if $(G, \kappa)$ has a $DP$-Nash subgraph for which $B$ is the $D$-set.

A natural question to ask is whether every capacitated graph possesses a Nash subgraph. 
\cite{GerkeGHN} confirmed the existence by proving the following:

\begin{theorem} \label{all_PD}
Every capacitated graph $(G,\kkk)$ has a Nash subgraph.
\end{theorem}

Since every capacitated graph has a Nash subgraph, Theorem \ref{all_PD} also implies that every capacitated graph has a $D$-set.
We provide a proof of Theorem~\ref{all_PD} in Appendix \ref{sec:AppExistence} to make this paper self-contained.

Given a capacitated graph $(G,\kkk)$, let us consider a few examples of $DP$-Nash subgraphs and $D$-sets.
If $\kappa(v)=0$ for every $v\in V$ then there is only one $DP$-Nash subgraph with $D$-set $V$ and empty $P$-set.
If $\kappa(v)=1$ for every $v\in V$ then every $DP$-Nash subgraph is a spanning vertex-disjoint collection of stars, each with at least two vertices.
If $\kappa(v)=d_G(v)$ for every $v\in V$ then the $D$-set of each $DP$-Nash subgraph of $(G, \kappa)$ is a maximal independent set of $G$.
It is well-known that a vertex set is maximal independent if and only if it is independent dominating. 
Since finding both a maximum size independent dominating set and a minimum size independent dominating set are both {\sf NP}-hard \cite{CornellP84}, so are the problems of finding a $D$-set of maximum and minimum size.
For more information on the complexity of independent domination, see \cite{GoddardH13}.

Below in Figure~\ref{fig:1} we present three capacitated graphs, $(G_1,\kkk_1)$, $(G_2,\kkk_2)$, and $(G_3,\kkk_3)$.
In each capacitated graph, the number adjacent to a vertex represents the capacity of that vertex.
Vertices labelled with an $x$ have capacity equal to their degree while vertices labelled with a $y$ have capacity strictly less than their degree and have a neighbour labelled with an $x$.
Any remaining vertices are labelled with a $z$.
(This labelling system is formally defined in \eqref{eq:XYZ}.)


\begin{figure}[hbt!]
\centering
\tikzstyle{vertexX}=[circle,draw, top color=gray!10, bottom color=gray!70, minimum size=10pt, scale=0.9, inner sep=0.5pt]
\tikzstyle{vertexY}=[circle,draw, top color=black!10, bottom color=gray!70, minimum size=15pt, scale=0.8, inner sep=0.4pt]
\begin{tikzpicture}[scale=0.3]
\draw (5,-1.5) node {{\small $(G_1,\kkk_1)$}};
\draw (1.0,1.0) node {{\small $1$}};
\draw (1.0,6.0) node {{\small $1$}};
\draw (9.0,6.0) node {{\small $1$}};
\draw (9.0,1.0) node {{\small $1$}};
\draw (5.0,4.9) node {{\small $3$}};
\node (x1) at (2.5,1.0) [vertexY] {$x_{4}$}; 
\node (x2) at (2.5,6.0) [vertexY] {$x_{1}$};
\node (x3) at (7.5,6.0) [vertexY] {$x_{2}$}; 
\node (x4) at (7.5,1.0) [vertexY] {$x_{3}$}; 
\node (y) at (5.0,3.5) [vertexY] {$y_{1}$};
\draw [thick] (y) -- (x1); 
\draw [thick] (y) -- (x2); 
\draw [thick] (y) -- (x3); 
\draw [thick] (y) -- (x4); 
\end{tikzpicture} \hfill 
\begin{tikzpicture}[scale=0.3]
\draw (5,-1.5) node {{\small $(G_2,\kkk_2)$}};
\draw (1.5,1.0) node {{\small $3$}};
\draw (8.5,1.0) node {{\small $1$}};
\draw (1.0,7.4) node {{\small $0$}};
\draw (5.0,7.4) node {{\small $0$}};
\draw (9.0,7.4) node {{\small $0$}};
\node (x) at (3.0,1.0) [vertexY] {$x$}; 
\node (z) at (7.0,1.0) [vertexY] {$z$}; 
\node (y1) at (1.0,6.0) [vertexY] {$y_{1}$}; 
\node (y2) at (5.0,6.0) [vertexY] {$y_{2}$}; 
\node (y3) at (9.0,6.0) [vertexY] {$y_{3}$}; 
\draw [thick] (x) -- (y1); 
\draw [thick] (x) -- (y2); 
\draw [thick] (x) -- (y3); 
\draw [thick] (z) -- (y1); 
\draw [thick] (z) -- (y2); 
\draw [thick] (z) -- (y3); 
\end{tikzpicture} \hfill
%
\begin{tikzpicture}[scale=0.3]
\draw (3.8,-1.5) node {{\small $(G_3,\kkk_3)$}};
\draw (1.0,6.0) node {{\small $2$}};
\draw (6.887,6.0) node {{\small $2$}};
\draw (8.33,3.5) node {{\small $2$}};
\draw (6.887,1.0) node {{\small $2$}};
\draw (1.0,1.0) node {{\small $2$}};
\draw (-0.43,3.5) node {{\small $2$}};
\node (z1) at (2.5,6.0) [vertexY] {$z_{1}$};
\node (z2) at (5.387,6.0) [vertexY] {$z_{2}$}; 
\node (z3) at (6.83,3.5) [vertexY] {$z_{3}$};
\node (z4) at (5.387,1.0) [vertexY] {$z_{4}$}; 
\node (z5) at (2.5,1.0) [vertexY] {$z_{5}$}; 
\node (z6) at (1.057,3.5) [vertexY] {$z_{6}$};
\draw [thick] (z1) -- (z2); 
\draw [thick] (z1) -- (z3); 
\draw [thick] (z1) -- (z4); 
\draw [thick] (z1) -- (z5); 
\draw [thick] (z1) -- (z6); 
\draw [thick] (z2) -- (z3); 
\draw [thick] (z2) -- (z4); 
\draw [thick] (z2) -- (z5); 
\draw [thick] (z2) -- (z6); 
\draw [thick] (z3) -- (z4); 
\draw [thick] (z3) -- (z5); 
\draw [thick] (z3) -- (z6); 
\draw [thick] (z4) -- (z5); 
\draw [thick] (z4) -- (z6); 
\draw [thick] (z5) -- (z6); 
\end{tikzpicture} \hfill
\caption{Three capacitated graphs $(G_1,\kkk_1)$, $(G_2,\kkk_2)$, and $(G_3,\kkk_3)$.}\label{fig:1}
\end{figure}

We will refer to the three capacitated graphs in Figure~\ref{fig:1} above on multiple occasions throughout the paper.
This is because, as we will show later, capacitated graph $(G_1,\kkk_1)$ has a unique Nash subgraph (and hence a unique $D$-set), capacitated graph $(G_2,\kkk_2)$ has three Nash subgraphs all with the same $D$-set, while capacitated graph $(G_3,\kkk_3)$ has multiple Nash subgraphs some of which have different $D$-sets.

In the rest of the paper, 
we will often write $G$ instead of $(G,\kappa)$ when the capacity function $\kappa$ is clear from the context.
We will often omit the subscript $G$ in $N_G(x)$ and $d_G(x)$ when the graph $G$ under consideration is clear from the context. 
We will often shorten the term $DP$-Nash subgraph to Nash subgraph. 

In the rest of this section, we provide two simple assumptions for the rest of the paper
which will allow us to simplify some of our proofs. 
In both assumptions,  $(G,\kappa)$ is a capacitated graph.

\begin{assumption}\label{ass:k<=d}
For all $u \in V$ we have $\kk{u} \leq d(u)$.
\end{assumption}
%
%

Assumption~\ref{ass:k<=d} states that every vertex has capacity no greater than its degree.
This assumption does not change the set of $DP$-Nash subgraphs of any capacitated graph as if $\kk{u} > d(u)$ we may let $\kk{u} = d(u)$ without changing $\min\{\kk{u},d(u)\}$.
Due to this assumption, we can simplify the definition of a $DP$-Nash subgraph in Definition~\ref{def:DPNash} by replacing the qualifier ``for every $x\in D$, $d_H(x)=\min\{d_G(x),\kappa(x)\}$'', with ``for every $x\in D$, $d_H(x)=\kappa(x)$''.
We note that Assumption~\ref{ass:k<=d} may not hold for a subgraph of $(G, \kappa)$ if the subgraph uses the same capacity function $\kappa$ restricted to its vertices.



\begin{assumption}\label{ass:notBothZero}
If $uv$ is an edge in $G$, then  $\kk{u} >0$ or $\kk{v} >0$
\end{assumption}

Assumption \ref{ass:notBothZero} states that no two vertices with capacity zero are adjacent.
This assumption does not change our problem due to the following result showing that any edge for which both end-vertices have capacity zero can effectively be ignored and treated as though it is not there.

\begin{proposition}
  Let $G^*$ be obtained from $G$ by deleting all edges $uv$ with $\kk{u}=\kk{v}=0$. Then
  $(D,P;E')$ is a Nash subgraph of $(G,\kkk)$ if and only if $(D,P;E')$ is a Nash subgraph of $(G^*,\kkk)$.
\end{proposition}
\begin{proof}
Let $uv$ be any edge in $G$ with $\kk{u}=\kk{v}=0$ and let $(D,P;E')$ be a Nash subgraph of $(G,\kkk)$.
Note that $uv \not\in E'$ as if $u \in D$ then $E'$ contains no edge incident with $u$ and if
$v\in D$ then $E'$ contains no edge incident with $v$ and if $u,v \in P$ then $E'$ does not contain the edge $uv$.
This implies $(D,P;E')$ is a Nash subgraph of $(G^*,\kkk)$.

Conversely if $(D,P;E')$ is a Nash subgraph of $(G^*,\kkk)$ then $(D,P;E')$ is a Nash subgraph of $(G,\kkk)$ as both graphs have the same capacity function and 
$G^*$ is a spanning subgraph of $G.$
\end{proof}

\section{Characterisation of capacitated Graphs with Unique $D$-set}\label{sec:char}

In this section we will fully characterise capacitated graphs that possess a unique $D$-set (see Theorem~\ref{thm:charDsetUnique}).
This characterisation will prove useful when we come to addressing our problems {\sc Nash Subgraph Uniqueness} and {\sc $D$-set Uniqueness} in Sections \ref{sec:DPs} and \ref{sec:Ds} respectively.

We begin by introducing some definitions and additional notation that will prove useful.
For a set $F$ of edges of a graph $H$ and a vertex $x$ of $H$, $N_F(x)=\{y\in V(H) \mid xy\in F\}$ and $d_F(x)=|N_F(x)|.$ 
For a vertex set $Q$ of a graph $H$, $N_H(Q)=\bigcup_{x\in Q}N_H(x).$

For a given capacitated graph $(G, \kkk)$ we define subsets of vertices $X(G,\kkk{})$, $Y(G,\kkk{})$ and $Z(G,\kkk{})$ as follows.
(If $(G,\kkk{})$ is clear from the context these sets will be denoted by $X$, $Y$ and $Z$, respectively.)

\begin{equation}\label{eq:XYZ}
\begin{array}{rcccl}
X & = &  X(G,\kkk{}) & := & \{ x \; | \; \kk{x} = d(x) \} \\
Y & = &  Y(G,\kkk{}) & := &N(X)\setminus X \\
Z & = &  Z(G,\kkk{}) & := & V(G) \setminus (X \cup Y) \\
\end{array}
\end{equation}

In words, $X(G,\kkk{})$ is the subset of vertices with capacity equal to degree, $Y(G,\kkk{})$ is the subset of vertices with capacity less than degree and that have at least one neighbour with capacity equal to degree capacity, and  $Z(G,\kkk{})$ are the remaining vertices.
The use of the symbols $x$, $y$ and $z$ to label vertices in the capacitated graphs of Figure~\ref{fig:1} should now be clear.


The following result considers $D$-sets from the ``local'' perspective of an individual vertex.
Specifically it shows that whether a vertex $u$ is part of some $D$-set is determined by how many vertices in the neighbourhood of $u$ have capacity equal to their degree, i.e., are contained in the subset of vertices $X(G,\kappa)$ as defined in \eqref{eq:XYZ}.


\begin{lemma} \label{lem_vertex_u}
Let $u \in V(G)$ and let $X=X(G,\kappa)$. If $|N_G(u) \cap X| \leq \kk{u}$ then there exists a Nash subgraph $(D,P;E)$ of $(G, \kappa)$ where
$u \in D$ and $N_G(u) \cap X \subseteq P$.
Furthermore, if $|N_G(u) \cap X| < \kk{u}$ and $w \in N_G(u) \setminus X$ then there exists a Nash subgraph $(D,P;E)$ of $(G, \kappa)$ where
$u \in D$ and $\{w\} \cup (N_G(u) \cap X) \subseteq P$. 
\end{lemma}
\begin{proof}
Let $u \in V(G)$ such that $|N_G(u) \cap X| \leq \kk{u}$. 
Recall that by Assumption~1 we have $\kk{v} \leq d_G(v)$ for all $v \in V(G)$.
Let $E^*$ denote an arbitrary set of $\kk{u}$ edges incident
with $u$, such that $N_G(u) \cap X \subseteq N_{E^*}(u)$. Let $T'= N_{E^*}(u)$, $G' = G \setminus (\{u\} \cup T')$ and 
 $(P',D';E')$ a Nash subgraph of $G'$, which exists by Theorem~\ref{all_PD}.  
Let $P=P' \cup  T'$ and let $D=D' \cup \{u\}$.

Initially let $\hat{E} = E' \cup E^*$. Clearly every vertex in $P$ has at least one edge into $D$. Now let $v \in D$ be arbitrary.  
If $d_{\hat{E}}(v) \not= \kk{v}$ (recall that $\kk{v} \leq d_G(v)$) 
then we observe that $v \in D'$ and $$d_{\hat{E}}(v) = d_{E'}(v) = \min\{d_{G'}(v),\kk{v}\} < \kk{v}.$$
Since $v$ either has no edge to $u$ or does not lie in $X$, observe that we can add $\kk{v} - d_{G'}(v)$ edges to $\hat{E}$ between $v$ and $T'$
resulting in $d_{\hat{E}}(v) = \kk{v}$. After doing the above for every $x \in D$ we obtain a 
Nash subgraph $(D,P;\hat{E})$ of $G$ with the desired properties. This completes our proof of the case $|N_G(u) \cap X| \leq \kk{u}.$ 
The same proof can be used for the case $|N_G(u) \cap X| < \kk{u}$ if we choose $E^*$ such that $w\in T'.$ 
\end{proof}

Lemma~\ref{lem_XZ_indep} below is the first step in how the subsets of vertices $X, Y$, and $Z$ as defined in \eqref{eq:XYZ} relate to the issue of unique $D$-sets.
It provides a simple sufficient condition, which is that if $X \cup Z$ is not independent in $G$, that the capacitated graph $(G, \kkk{})$ will have multiple $D$-sets.

\begin{lemma} \label{lem_XZ_indep}
If $X \cup Z$ is not an independent set in $(G, \kappa)$, then there exist \PDdecomps{} of $G$ with different $D$-sets.
If $X \cup Z$ is an independent set in $G$ then there exists a \PDdecomp{} $(D,P;E')$ of $(G, \kappa)$ 
where $D=X \cup Z$ and $P=Y$.
\end{lemma}
\begin{proof}
First assume that $X \cup Z$ is not independent in $G$ and that $uv$ is an edge where $u,v \in X \cup Z$. 
By the definition of $Z$ we have that $u,v \in X$ or $u,v \in Z$.

First consider the case when $u,v \in X$. Note that $|N_G(u) \cap X| \leq d(u) = \kk{u}$, which by Lemma~\ref{lem_vertex_u}
implies that there is a Nash subgraph $(D',P';E')$ in $G$ where $u \in D'$ and $v \in P'$. Analogously, we can obtain a 
Nash subgraph $(D'',P'';E'')$ in $G$ where $v \in D''$ and $u \in P''$, which implies that there exist Nash subgraphs of $G$ with different $D$-sets, as
desired.

We now consider the case when $u,v \in Z$. As $uv$ is an edge in $G$ we may without loss of generality assume that $\kk{v} \geq 1$ (by Assumption~2).
As $|N_G(v) \cap X| = 0 < \kk{v}$,  Lemma~\ref{lem_vertex_u}
implies that there is a Nash subgraph $(D',P';E')$ in $G$ where $v \in D'$ and $u \in P'$. 
As  $|N_G(u) \cap X| = 0 \leq \kk{u}$,  Lemma~\ref{lem_vertex_u}
implies that there is a Nash subgraph $(D'',P'';E'')$ in $G$ where $u \in D''$.
As there exist Nash subgraphs where $u \in P'$ and where $u \in D''$, we are done in this case.

Now let $X \cup Z$ be independent in $G$ and let $P=Y$ and $D=X \cup Z$. Let $E'$ contain all edges between 
$X$ and $Y$ as well as any $\kk{z}$ edges from $z$ to $P$ for all $z \in Z$. As $Y = N(X) \setminus X$ we conclude that 
$(D,P;E')$ is a Nash subgraph of $(G, \kappa)$. 
\end{proof}

To state our characterisation result for capacitated graphs possessing a unique $D$-set, we need some additional definitions and two properties.

Given a capacity $\kappa$ on a graph $G$, for any subset $U \subseteq V(G)$ let $U^{\kkk{}}$ denote a set of vertices obtained from $U$ by replacing each vertex, $u \in U$, by $\kk{u}$ copies.
Note that if $\kk{u}=0$ then the vertex $u$ is not in $U^{\kkk{}}$ and
$|U^{\kkk{}}| = \sum_{u \in U} \kk{u}$.


\begin{definition}
Given a capacitated graph $(G, \kappa)$, we define its {\it auxiliary graph}, $G^{\aux}(G, \kappa)$ (or simply $G^{\aux}$ when $(G, \kappa)$ is understood), as the bipartite graph with partite sets $R' = X \cup Z$ and $Y' = Y^{\kkk{}}$.
For a vertex $y\in Y,$ there is an edge from a copy of $y$ to $r \in R'$ in $G^{\aux}$ if and only if
there is an edge from $y$ to $r$ in $G$. 
\end{definition}

Let $Y^{\kkk{}>0} \subseteq Y$ consist of all vertices $y\in Y$ with $\kk{y}>0$.
For every set $\emptyset \not= W \subseteq Y^{\kkk{}>0}$ let $$L(W) = \{ v \in X \cup Z \; | \; |N(v) \cap W| > d(v)-\kk{v} \}.$$

Let us return to the three examples of capacitated graphs in Figure~\ref{fig:1} and consider $G^{\aux}$ for each.
We note that $G^{\aux}(G_2,\kkk_2)$ is simply the two isolated vertices $x$ and $z$, and that $G^{\aux}(G_3,\kkk_3) = G_3$.

In Figure~\ref{fig:2} below, the graph on the left is the capacitated graph $(G_1,\kkk_1)$ repeated and the graph on the right is its auxiliary graph $G^{\aux}(G_1,\kkk_1)$.
Since $\kkk_{1}(y) = 3$ in $G^{\aux}(G_1,\kkk_1)$ we have that $\abs{Y'} = \abs{Y^{\kkk_{1}}} = 3$ in $G^{\aux}$.


\begin{figure}[hbt!]
\centering
\tikzstyle{vertexX}=[circle,draw, top color=gray!10, bottom color=gray!70, minimum size=10pt, scale=0.9, inner sep=0.5pt]
\tikzstyle{vertexY}=[circle,draw, top color=black!10, bottom color=gray!70, minimum size=15pt, scale=0.8, inner sep=0.4pt]
\hfill
\begin{tikzpicture}[scale=0.3]
\draw (5,-1.5) node {{\small $(G_1,\kkk_1)$}};
\draw (1.0,1.0) node {{\small $1$}};
\draw (1.0,6.0) node {{\small $1$}};
\draw (9.0,6.0) node {{\small $1$}};
\draw (9.0,1.0) node {{\small $1$}};
\draw (5.0,4.9) node {{\small $3$}};
\node (x1) at (2.5,1.0) [vertexY] {$x_{4}$}; 
\node (x2) at (2.5,6.0) [vertexY] {$x_{1}$};
\node (x3) at (7.5,6.0) [vertexY] {$x_{2}$}; 
\node (x4) at (7.5,1.0) [vertexY] {$x_{3}$}; 
\node (y) at (5.0,3.5) [vertexY] {$y$};
\draw [thick] (y) -- (x1); 
\draw [thick] (y) -- (x2); 
\draw [thick] (y) -- (x3); 
\draw [thick] (y) -- (x4); 
\end{tikzpicture} \hfill 
\begin{tikzpicture}[scale=0.3]
\draw (5,-1.5) node {{\small $G^{\aux}(G_1,\kkk_1)$}};
\node (x1) at (-1.0,1.0) [vertexY] {$x_{1}$}; 
\node (x2) at (3.0,1.0) [vertexY] {$x_{2}$}; 
\node (x3) at (7.0,1.0) [vertexY] {$x_{3}$}; 
\node (x4) at (11.0,1.0) [vertexY] {$x_{4}$}; 
\node (y1) at (1.0,6.0) [vertexY] {$y_{1}$}; 
\node (y2) at (5.0,6.0) [vertexY] {$y_{2}$}; 
\node (y3) at (9.0,6.0) [vertexY] {$y_{3}$}; 
\draw [thick] (x1) -- (y1); 
\draw [thick] (x1) -- (y2); 
\draw [thick] (x1) -- (y3); 
\draw [thick] (x2) -- (y1); 
\draw [thick] (x2) -- (y2); 
\draw [thick] (x2) -- (y3); 
\draw [thick] (x3) -- (y1); 
\draw [thick] (x3) -- (y2); 
\draw [thick] (x3) -- (y3); 
\draw [thick] (x4) -- (y1); 
\draw [thick] (x4) -- (y2); 
\draw [thick] (x4) -- (y3); 
\end{tikzpicture} \hfill \hfill 
\caption{Capacitated graph $(G_1,\kkk_1)$ and its auxiliary graph $G^{\aux}$.}\label{fig:2}
\end{figure}

As we will soon show, determining whether or not a capacitated graph $(G,\kkk)$ has a unique $D$-set turns out to be equivalent to whether or not there are various matchings of particular sizes in $G^{\aux}$.
To describe these types of matchings we define the following two properties $M^*(G,\kkk{})$ and $M^{**}(G,\kkk{})$.

\begin{description}
 \item[Property $\mathbf{M^*(G,\kkk{})}$:] holds if for every set $\emptyset \not= W \subseteq Y^{\kkk{}>0}$ there is no matching from $L(W)$ to $W^{\kkk{}}$ of size $|L(W)|$
in $G^{\aux}.$ 
 \item[Property $\mathbf{M^{**}(G,\kkk{})}$:] holds if for every set $\emptyset \not= W \subseteq Y^{\kkk{}>0}$ we have $|L(W)|>|W^{\kkk{}}|$.
\end{description}

The next theorem fully characterises capacitated graphs possessing a unique $D$-set.

\begin{theorem}\label{thm:charDsetUnique}
Let $(G, \kkk{})$ be a capacitated graph.
If vertex subset $X \cup Z$ is not independent in $G$ then $(G,\kappa)$ has at least two different $D$-sets.
If vertex subset $X \cup Z$ is independent in $G$ then the following three statements are equivalent:
\begin{description}
\item[(a)] $(G,\kkk{})$ has a unique $D$-set;
\item[(b)] $M^*(G,\kkk{})$ holds;
\item[(c)] $M^{**}(G,\kkk{})$ holds.
\end{description}
\end{theorem}

\begin{proof}
The case of  $X \cup Z$ being not independent  follows from Lemma~\ref{lem_XZ_indep}. We will therefore assume that $X \cup Z$ is independent in $G$
and prove the rest of the theorem by showing that (a) $\Rightarrow$ (b) $\Rightarrow$ (c) $\Rightarrow$ (a).
The following three claims complete the proof.

\2

{\bf Claim A: (a) $\Rightarrow$ (b).}  

\2

{\em Proof of Claim A:} Suppose that (a) holds but (b) does not.
As (b) is false, $M^*(G,\kkk{})$ does not hold, which implies that there exists a $\emptyset \not= W \subseteq Y^{\kkk{}>0}$
such that there is a matching, $M$, from $L(W)$ to $W^{\kkk{}}$ of size $|L(W)|$ in $G^{\aux}$.

Let $D_1 = W$, $P_1 = L(W)$ and $G_2 = G - (P_1 \cup D_1)$.
Let $(D_2,P_2;E_2')$ be a \PDdecomp{} of $G_2$, which exists by Theorem~\ref{all_PD}. We will now prove the following six subclaims.

\2

{\bf Subclaim~A.1:} For every $y \in Y,$ we have $|N(y) \cap X| \geq \kk{y}+1$.

\2

{\em Proof of Subclaim A.1:} Assume that Subclaim~A.1  is false and there exists a vertex $y \in Y$ such that $|N(y) \cap X| \leq \kk{y}$.
By Lemma~\ref{lem_vertex_u}, there exists a Nash subgraph $(D',P';E')$ in $G$ where $y \in D'$.
By Lemma~\ref{lem_XZ_indep}, there exists a Nash subgraph $(D'',P'';E'')$ in $G$ where $y \in P''$ (as $P''=Y$).
Therefore (a) is false, a contradiction. $\diamond$

\2

{\bf Sublaim~A.2:} If $u \in D_1 = W$ then $N(u) \cap X \subseteq P_1$. 
Furthermore, $|N(u) \cap P_1| \geq \kk{u}+1$.

\2

{\em Proof of Subclaim~A.2:} Let $u \in D_1$ (and therefore $u \in W$) 
be arbitrary and let $r \in N(u) \cap X$ be arbitrary. We will show that $r \in P_1$,
which will prove the first part of the claim.
As $r \in X$ we have $d_G(r)=\kk{r}$.  This implies that $|N(r) \cap W| \geq 1 > 0 = d_G(r) - \kk{r}.$ Hence, $r \in L(W) = P_1$ 
as desired. 

We now prove the second part of Subclaim~A.2.
Since $u \in Y,$ Subclaim~A.1 implies that $|N(u) \cap X| \geq \kk{u}+1$. 
As every vertex in $N(u) \cap X$ also belongs to $P_1,$ we have $|N(u) \cap P_1| \geq \kk{u}+1$. $\diamond$

\2

{\bf Subclaim~A.3:} There exists a Nash subgraph $(D_1,P_1;E_1')$ of $\induce{G}{D_1 \cup P_1}$.

\2

{\em Proof of Subclaim~A.3:}  $P_1$ and $D_1$ were defined earlier so we will now define $E_1'$.  Let all edges of the matching 
$M$ belong to $E_1'$.  That is if $u'v \in M$ and $u' \in V(G^{\aux})$ is a copy of $u \in V(G)$, then add the edge $uv$ to $E_1'$.
We note that every vertex in $P_1$ is incident to exactly one of the edges added so far and every vertex $u\in D_1$ is incident 
to at most $\kk{u}$ such edges. By Subclaim~A.2 we can add further edges between $P_1$ and $D_1$ such that 
every vertex $u\in D_1$ is incident with exactly $\kk{u}$ edges from $E_1'$. $\diamond$ 

\2

{\bf Subclaim~A.4:} Every $u \in (X \cup Z) \setminus L(W)$ has at least $\kk{u}$ neighbours in $Y \setminus W$.

\2

{\em Proof of Subclaim~A.4:} As $u \not\in L(W)$ we
have that $|N_G(u) \cap W| \leq d(u)-\kk{u}$. This implies that
$|N_G(u) \setminus W| \geq \kk{u}$.
As $X \cup Z$ is independent this implies that $u$ has at least $\kk{u}$ neighbours in $Y \setminus W,$ as desired. $\diamond$

\2
Recall that $(D_2,P_2;E_2')$ is a Nash subgraph of $G_2$ and by Subclaim~A.3, $(D_1,P_1;E_1')$ is a \PDdecomp{} of $\induce{G}{D_1 \cup P_1}$.

\2

{\bf Subclaim~A.5:} There exists a Nash subgraph $(P_1 \cup P_2,D_1 \cup D_2,E_1' \cup E_2' \cup E^*)$ of $G$
for some $E^*$.

\2

{\em Proof of Subclaim~A.5:} 
Let $P=P_1 \cup P_2$, $D=D_1 \cup D_2$ and $E'=E_1' \cup E_2'$.
Clearly every vertex in $P$ is incident with an edge in $E'$.  

First consider a vertex $u \in D_1$. 
By Subclaim~A.2, $u$ has at least $\kk{u}+1$ neighbours in $P_1$ in $G$. Therefore, 
by Subclaim~A.3, $u$ is incident with exactly $\kk{u}$ edges of $E_1'$ and so also with $\kk{u}$ edges of $E'$.

Now consider $u \in D_2$. Note that $u$ is incident with $\min\{\kk{u},d_{G_2}(u)\}$ edges of $E_2'$.
If $u \in (X \cup Z) \setminus L(W)$ then by Subclaim~A.4, $\min\{\kk{u},d_{G_2}(u)\} = \kk{u}$ 
implying that $u$ is incident with exactly $\kk{u}$ edges of $E'$ as desired. 
We may therefore assume that $u \not\in (X \cup Z) \setminus L(W)$, which implies that $u \in Y \setminus W$.
By Subclaim~A.1, $u$ has at least $\kk{u}+1$ neighbours in $X$ in $G$. Therefore, $d_{G_2}(u) + |N(u) \cap P_1| \geq \kk{u}+1$
(as every edge from $u$ to $X$ is counted in the sum on the left hand side of the inequality).
Thus, if $\min\{\kk{u},d_{G_2}(u)\} < \kk{u}$, then we can add edges from $u$ to $P_1$ to $E'$ until 
$u$ is incident with exactly $\kk{u}$ edges of $E'$. Continuing the above process for all $u$ and letting $E^*$ be the added 
edges, we obtain the claimed result. $\diamond$

\2

{\bf Subclaim~A.6:} Claim A holds. 

\2

{\em Proof of Subclaim~A.6:} By Lemma~\ref{lem_XZ_indep} there exists a \PDdecomp{}, $(D,P,E')$,
with $D=X \cup Z$ and $P=Y$. By Subclaim~A.5, 
there exists a \PDdecomp{} of $G$ where some vertices of $Y$ belong to $D$, contradicting the fact that (a) holds.
This completes the proof of Subclaim~A.6, and therefore also of Claim~A. $\diamond$

\2
\2

{\bf Claim B: (b) $\Rightarrow$ (c).} 

\2

{\em Proof of Claim B:} Suppose that (b) holds but (c) does not.
As (c) does not hold there exists a $\emptyset \not= W \subseteq Y^{\kkk{}>0}$
such that $|L(W)| \leq |W|$. Assume that $W$ is chosen such that $|W|$ is minimum possible with this property.
As (b) holds there is no matching between $W^{\kkk{}}$ and $L(W)$ in $G^{\aux}$ saturating every vertex of $L(W)$.
By Hall's Theorem, this implies that there exists a set $S \subseteq L(W)$ such that $|N_{G^{\aux}}(S)|<|S|$.
Note that $N_G(S) \subseteq W$ such that $N_{G^{\aux}}(S)$ contains exactly the copies of $N_G(S)$.
Note that $|N_G(S)| \leq |N_{G^{\aux}}(S)|<|S|$ as $W \subseteq Y^{\kkk{}>0}$. 
Let $W' = W \setminus N_G(S)$. 

By definition we have $$L(W') = \{ x \in X  \cup Z \; | \; |N_G(x) \cap W'| > d(x)-\kk{x} \}.$$
We will now prove the following subclaim.

\2

{\bf Subclaim~B.1:} $L(W') \subseteq L(W) \setminus S$.

\2

{\em Proof of Subclaim~B.1:} Let $u \in L(W')$ be arbitrary and note that 
$|N_G(u) \cap W| \geq |N_G(u) \cap W'| > d(u)-\kk{u}.$
This implies that $u \in L(W)$. 

We will now show that $u \not\in S$.  If $u \in S$, 
then $N(u) \subseteq N(S)$, so $u$ has no neighbours in $W' = W \setminus N(S)$.
Therefore, $|N_G(u) \cap W'| = 0$, and as we assumed that $d(v) \geq \kk{v}$ for all $v \in V(G)$, the following holds

\[
|N_G(u) \cap W'| = 0 \leq d(u)-\kk{u}.
\]

Therefore, $u \not\in L(W')$, a contradiction.  This implies that $u \not\in S$ and therefore 
$L(W') \subseteq L(W) \setminus S$. $\diamond$

\2

By Subclaim~B.1, we have that $|L(W')| \leq |L(W)|-|S| < |L(W)| - |N_G(S)| = |W'|$.
This contradicts the minimality of $|W|$, and therefore completes the proof of Claim~B.

\2
\2

{\bf Claim C: (c) $\Rightarrow$ (a).} 

\2

{\em Proof of Claim C:}
Suppose that (c) holds but (a) does not.
By Lemma~\ref{lem_XZ_indep} and the fact that (a) does not hold, there exists a Nash subgraph $(D,P;E')$ of $G$ such that  $D \not= X \cup Z$. 

If $Y^{\kkk{}>0} \subseteq P$, then no vertex of $X \cup Z$ can belong to $P$ as it would have no edge to $D$ (as $X \cup Z$ is independent).
Therefore, $X \cup Z \subseteq D$ in this case. Due to the definition of $X$ (and $Y$) and the fact that $X \subseteq D$, we have that $Y \subseteq P$, 
which implies that $D=X \cup Z$ and $Y=P$, which is a contradiction to our assumption that $D \not= X \cup Z$.

So we may assume that $Y^{\kkk{}>0} \not\subseteq P$. This implies that $Y^{\kkk{}>0} \cap D \not= \emptyset$.
Let $W = Y^{\kkk{}>0} \cap D$.  We now prove the following subclaim.

\2

{\bf Subclaim~C.1:} $L(W) \subseteq P$.

\2

{\em Proof of Subclaim~C.1:} Let $w \in L(W)$ be arbitrary.
Hence, $|N_G(w) \cap W| > d(w) - \kk{w}$.
If $w \in D$, then $w$ has at least $\kk{w}$ neighbours in $P$ in $G$.
By the above it has at least $d(w) - \kk{w} + 1$ neighbours in $W \subseteq D$, contradicting the
fact that $w$ has $d(w)$ neighbours.
This implies that $w \not\in D$.
Therefore, $w \in P$ and as $w \in L(W)$ is arbitrary, we must have $L(W) \subseteq P$.
$\diamond$

\2

We now return to the proof of Claim~C. 
Recall that $X \cup Z$ is independent and $L(W) \subseteq X \cup W$ and every vertex in $P$ has at least one edge to $D$ in $E'$.
By Subclaim~C.1, $L(W) \subseteq P$, which implies that there are at least $|L(W)|$ edges from $L(W)$ to $W$, 
as $W = Y^{\kkk{}>0} \cap D$. As there are at most $\theta = \sum_{w \in W} \kk{w}$ edges from $W$ to $L(W)$ we must have
$
|L(W)| \leq \theta = \sum_{w \in W} \kk{w} = |W^{\kkk{}}|.
$

The above is a contradiction to (c).
This completes the proof of Claim~C and therefore also of the theorem. 
\end{proof}

We immediately have the following:

\begin{corollary}\label{mainCOR}
All Nash subgraphs of $(G,\kappa)$ have the same $D$-set $D$ if and only if $D=X \cup Z$, $D$ is independent in $G,$ and $M^{**}(G,\kkk{})$ holds.
\end{corollary}

Note that if $\kk{x}=0$ for all $x \in G$ then  $X=V(G)$ and $Y=\emptyset$.
In this case $M^{**}(G,\kkk{})$ vacuously holds and there is a unique Nash subgraph of $(G, \kappa)$ with $D$-set $V$ and empty $P$-set.

We conclude this section with the following result relating capacitated graphs that admit a unique $D$-set to maximum independent sets of the underlying graph.

{
\begin{proposition}\label{prop:uniqueDmaxInd}
Let $(G, \kkk)$ be a capacitated graph with $\kk{v} > 0$ for all $v \in V(G)$.
If $(G, \kkk)$ has a unique $D$-set $D$, then $D$ is the unique maximum independent set of the graph $G$.
\end{proposition}
}
\begin{proof}
Let $X$, $Y$, and $Z$ be defined as in \eqref{eq:XYZ}, and let us suppose that $(G, \kkk)$ has a unique $D$-set, $D$.
By Corollary~\ref{mainCOR} above we have that $D = X \cup Z$ and that $D$ is independent.
Again by Corollary~~\ref{mainCOR}, we have that condition $M^{**}$ holds and so for every non-empty subset $W$ of $Y$ we have
\[
\abs{N(W)} \ge \abs{L(W)} > \abs{W^{\kkk}} \ge \abs{W}
\]
Now let $J$ be an independent set of $G$ where $W = J \cap Y$.
If $\abs{W} > 0$, then $\abs{J} \le \abs{W} + \abs{D - N(W)} = \abs{W}+\abs{D} - \abs{N_D(W)} < \abs{D}$.
Therefore if $\abs{W} > 0$, it must be that $J$ is strictly smaller than $D$.
It remains to consider that $\abs{W} = 0$.
But in this case it must be that $J = D$ or $\abs{J} < \abs{D}$.
This proves that not only is the $D$-set $D$ of capacitated graph $(G, \kkk)$ a maximum independent set of $G$, it is the unique maximum independent set of $G$.

%
%
\end{proof}

Note that the condition that every vertex has strictly positive capacity is necessary for Proposition~\ref{prop:uniqueDmaxInd} to hold.
As a counter example, consider the complete bipartite graph $K_{s,t}$ with partite sets $A$ and $B$, where $\abs{A} = s < t = \abs{B}$.
Now consider the capacitated graph $(K_{s,t}, \kkk{})$ with $\kk{x} = t$ for all $x \in A$ and $\kk{y} = 0$ for all $y \in B$.
Then $A$ is the unique $D$-set of $(K_{s,t}, \kkk{})$ whereas $B$ is the unique maximum independent set of $K_{s,t}$.

\section{Complexity of Uniqueness of Nash Subgraph}\label{sec:DPs}

In this section we prove that deciding whether a capacitated graph $(G,\kkk{})$ has a unique Nash subgraph can be done in polynomial time.

\begin{theorem}\label{thm:P}
{\sc Nash Subgraph Uniqueness} is in {\sf P}.
\end{theorem}

\begin{proof}
Let $(G,\kappa)$ be a capacitated graph, and let $X = X(G,\kkk{})$, $Y=Y(G,\kkk{})$ and
$Z = Z(G,\kkk{})$ be defined as in \eqref{eq:XYZ}.
If $X \cup Z$ is not independent then there exist distinct Nash subgraphs in $(G, \kappa)$ 
by Lemma~\ref{lem_XZ_indep}. So we may assume that $X \cup Z$ is independent.
By Lemma~\ref{lem_XZ_indep} there exists a Nash subgraph $(D,P;E')$ in $G$
where $D= X \cup Z$ and $P=Y$.  

If $Z \not= \emptyset$, then let $z \in Z$
be arbitrary.  
In $E'$ we may pick any $\kk{z}$ edges out of $z$, as every vertex in $Y$ has an edge to $X$ in $E'$.
As $d(z)>\kappa(z)$ we note that 
by picking different edges incident with $z$ we get distinct 
Nash subgraphs of $G$.  We may therefore assume that $Z = \emptyset$.

Recall the definition of $G^{\aux}$, which has partite sets $R' = X$ and $Y' = Y^{\kkk{}}$ (as $Z = \emptyset$).

We will now prove the following two claims which complete the proof of the theorem since the existence of a matching in $G^{\aux} - x$ saturating its partite set
$Y'$
 can be decided in polynomial time for every $x\in X.$

\2

{\bf Claim A:} {\em If for every $x \in X$ there exists a matching in $G^{\aux} - x$ saturating $Y'$ then
there is only one Nash subgraph in $G$.}

\2

{\em Proof of Claim~A:} We will first show that if the statement of Claim~A holds then 
 $M^{**}(G,\kkk{})$ holds. 
Suppose that $M^{**}(G,\kkk{})$ does not hold. 
This implies that there is a set $\emptyset \not= W \subseteq Y^{\kkk{}>0}$ such that
$|L(W)| \leq |W^{\kkk{}}|$.
Note that, as $Z = \emptyset$, we have $L(W)=N_G(W) \cap X$. As $W \not= \emptyset$ and $W \subseteq Y$,
we have that $N_G(W) \cap X \not= \emptyset$. Let $x \in N_G(W) \cap X$ be arbitrary.
Now the following holds. 
\[
|(N_G(W) \cap X) \setminus \{x\}| = |L(W)|-1 \leq |W^{\kkk{}}| -1 < |W^{\kkk{}}|
\]

This implies that there cannot be a  matching in $G^{\aux} - x$ saturating $Y'$, a contradiction.
Thus,  $M^{**}(G,\kkk{})$ must hold. 
By Corollary~\ref{mainCOR} we have that all Nash subgraphs must therefore have the same $D$-set.
By Lemma~\ref{lem_XZ_indep} we have that all Nash subgraphs $(D,P;E')$ must therefore have
$D=X$ and $P=Y$.  By the definition of $X$ we note that $E'$ must contain exactly the edges between 
$X$ and $Y$, and therefore there is a unique Nash subgraph in $G$. $\diamond$

\2

{\bf Claim B:} {\em If for some $x \in X$ there is no matching in $G^{\aux} - x$ saturating $Y',$ then
there are at least two distinct Nash subgraphs in $G$.}

\2

{\em Proof of Claim~B:} Let $x \in X$ be defined as in the statement of Claim~B.
By Hall's Theorem there exists a set $S' \subseteq Y'$ such that $|N_{G^{\aux}}(S') \setminus \{x\}|<|S'|$.
Let $S \subseteq Y$ be the set of vertices for which there is a copy in $S'$. 
Note that $(N_G(S) \cap X) \setminus \{x\} = N_{G^{\aux}}(S') \setminus \{x\}$ and 
$|S'| \leq |S^{\kkk{}}|$, which implies the following.
\begin{eqnarray*}
|N_G(S) \cap X| & \leq & |(N_G(S) \cap X) \setminus \{x\}| + 1 =  |N_{G^{\aux}}(S')  \setminus \{x\}| + 1\\
& < & |S'|+1 \leq |S^{\kkk{}}| +1.
\end{eqnarray*}
As all terms above are integers, this implies that $|N_G(S) \cap X| \leq |S^{\kkk{}}|$.
As $L(S) = N_G(S) \cap X$ by the definition of $L(S)$, we note that $|L(S)| \leq |S^{\kkk{}}|$ 
and therefore $M^{**}(G,\kkk{})$ does not hold, which by Corollary~\ref{mainCOR} implies 
that there are distinct Nash subgraphs in $G$ (even with distinct $D$-sets).
This completes the proof of Claim~B and therefore also of the theorem. 
\end{proof}

Let us return again to consider capacitated graph $(G_1,\kkk_1)$ and its auxiliary graph $G^{\aux}(G_1,\kkk_1)$ as depicted in Figure~\ref{fig:2}.
By Claim A in Theorem~\ref{thm:P}, $(G_1,\kkk_1)$ has a unique Nash subgraph if for every $x \in X$, there is a matching in $G^{\aux}(G_1,\kkk_1) - x$ that saturates $Y'$.
Figure~\ref{fig:3} below depicts each of the four possible variants, $G^{\aux}(G_1,\kkk_1) - x_i$ for $i = 1,\dots, 4$.

\begin{figure}[hbt!]
\centering
\tikzstyle{vertexX}=[circle,draw, top color=gray!10, bottom color=gray!70, minimum size=10pt, scale=0.9, inner sep=0.5pt]
\tikzstyle{vertexY}=[circle,draw, top color=black!10, bottom color=gray!70, minimum size=15pt, scale=0.8, inner sep=0.4pt]
\hfill
\begin{tikzpicture}[scale=0.2]
\draw (5,-2.0) node {{\small $G^{\aux} - x_1$}};
\node (x2) at (3.0,1.0) [vertexY] {$x_{2}$}; 
\node (x3) at (7.0,1.0) [vertexY] {$x_{3}$}; 
\node (x4) at (11.0,1.0) [vertexY] {$x_{4}$}; 
\node (y1) at (1.0,6.0) [vertexY] {$y_{1}$}; 
\node (y2) at (5.0,6.0) [vertexY] {$y_{2}$}; 
\node (y3) at (9.0,6.0) [vertexY] {$y_{3}$}; 
\draw [thick] (x2) -- (y1); 
\draw [thick] (x2) -- (y2); 
\draw [thick] (x2) -- (y3); 
\draw [thick] (x3) -- (y1); 
\draw [thick] (x3) -- (y2); 
\draw [thick] (x3) -- (y3); 
\draw [thick] (x4) -- (y1); 
\draw [thick] (x4) -- (y2); 
\draw [thick] (x4) -- (y3); 
\end{tikzpicture} \hfill
\begin{tikzpicture}[scale=0.2]
\draw (5,-2.0) node {{\small $G^{\aux} - x_2$}};
\node (x1) at (-1.0,1.0) [vertexY] {$x_{1}$}; 
\node (x3) at (7.0,1.0) [vertexY] {$x_{3}$}; 
\node (x4) at (11.0,1.0) [vertexY] {$x_{4}$}; 
\node (y1) at (1.0,6.0) [vertexY] {$y_{1}$}; 
\node (y2) at (5.0,6.0) [vertexY] {$y_{2}$}; 
\node (y3) at (9.0,6.0) [vertexY] {$y_{3}$}; 
\draw [thick] (x1) -- (y1); 
\draw [thick] (x1) -- (y2); 
\draw [thick] (x1) -- (y3); 
\draw [thick] (x3) -- (y1); 
\draw [thick] (x3) -- (y2); 
\draw [thick] (x3) -- (y3); 
\draw [thick] (x4) -- (y1); 
\draw [thick] (x4) -- (y2); 
\draw [thick] (x4) -- (y3); 
\end{tikzpicture} \hfill 
\begin{tikzpicture}[scale=0.2]
\draw (5,-2.0) node {{\small $G^{\aux} - x_3$}};
\node (x1) at (-1.0,1.0) [vertexY] {$x_{1}$}; 
\node (x2) at (3.0,1.0) [vertexY] {$x_{2}$}; 
\node (x4) at (11.0,1.0) [vertexY] {$x_{4}$}; 
\node (y1) at (1.0,6.0) [vertexY] {$y_{1}$}; 
\node (y2) at (5.0,6.0) [vertexY] {$y_{2}$}; 
\node (y3) at (9.0,6.0) [vertexY] {$y_{3}$}; 
\draw [thick] (x1) -- (y1); 
\draw [thick] (x1) -- (y2); 
\draw [thick] (x1) -- (y3); 
\draw [thick] (x2) -- (y1); 
\draw [thick] (x2) -- (y2); 
\draw [thick] (x4) -- (y1); 
\draw [thick] (x4) -- (y2); 
\draw [thick] (x4) -- (y3); 
\end{tikzpicture}
\hfill 
\begin{tikzpicture}[scale=0.2]
\draw (5,-2.0) node {{\small $G^{\aux} - x_4$}};
\node (x1) at (-1.0,1.0) [vertexY] {$x_{1}$}; 
\node (x2) at (3.0,1.0) [vertexY] {$x_{2}$}; 
\node (x3) at (7.0,1.0) [vertexY] {$x_{3}$}; 
\node (y1) at (1.0,6.0) [vertexY] {$y_{1}$}; 
\node (y2) at (5.0,6.0) [vertexY] {$y_{2}$}; 
\node (y3) at (9.0,6.0) [vertexY] {$y_{3}$}; 
\draw [thick] (x1) -- (y1); 
\draw [thick] (x1) -- (y2); 
\draw [thick] (x1) -- (y3); 
\draw [thick] (x2) -- (y1); 
\draw [thick] (x2) -- (y2); 
\draw [thick] (x2) -- (y3); 
\draw [thick] (x3) -- (y1); 
\draw [thick] (x3) -- (y2); 
\draw [thick] (x3) -- (y3); 
\end{tikzpicture} \hfill 
\caption{The collection of graphs $G^{\aux} - x_i$ for $i = 1,\dots, 4$.}\label{fig:3}
\end{figure}

As can be seen for each $i = 1,\dots, 4$ we have that $G^{\aux} - x_i$ is a complete bipartite graph with equal sized partite sets.
Since $Y'$ is one of the partite sets in each auxiliary graph, clearly there is a matching that saturates $Y'$ in each.
Therefore, by Claim A in Theorem~\ref{thm:P}, capacitated graph $(G_1,\kkk_1)$ has a unique Nash subgraph.
This Nash subgraph is depicted below in Figure~\ref{fig:4} with vertices in the $D$-set shaded in \blue{blue} and vertices in the $P$-set shaded in \red{red}.

\begin{figure}[hbt!]
\centering
\tikzstyle{vertexX}=[circle,draw, top color=gray!10, bottom color=gray!70, minimum size=10pt, scale=0.9, inner sep=0.5pt]
\tikzstyle{vertexY}=[circle,draw, top color=black!10, bottom color=gray!70, minimum size=15pt, scale=0.8, inner sep=0.4pt]
\tikzstyle{vertexYb}=[circle,draw, top color=blue!40, bottom color=blue!40, minimum size=15pt, scale=0.8, inner sep=0.4pt]
\tikzstyle{vertexYr}=[circle,draw, top color=red!40, bottom color=red!40, minimum size=15pt, scale=0.8, inner sep=0.4pt]
\hfill
\begin{tikzpicture}[scale=0.3]
\draw (5,-1.5) node {{\small $(G_1,\kkk_1)$}};
\draw (1.0,1.0) node {{\small $1$}};
\draw (1.0,6.0) node {{\small $1$}};
\draw (9.0,6.0) node {{\small $1$}};
\draw (9.0,1.0) node {{\small $1$}};
\draw (5.0,4.9) node {{\small $3$}};
\node (x1) at (2.5,1.0) [vertexY] {$x_{4}$}; 
\node (x2) at (2.5,6.0) [vertexY] {$x_{1}$};
\node (x3) at (7.5,6.0) [vertexY] {$x_{2}$}; 
\node (x4) at (7.5,1.0) [vertexY] {$x_{3}$}; 
\node (y) at (5.0,3.5) [vertexY] {$y$};
\draw [thick] (y) -- (x1); 
\draw [thick] (y) -- (x2); 
\draw [thick] (y) -- (x3); 
\draw [thick] (y) -- (x4); 
\end{tikzpicture} \hfill 
\begin{tikzpicture}[scale=0.3]
\draw (5,-1.5) node {{\small Nash subgraph}};
\node (x1) at (2.5,1.0) [vertexYb] {$x_{4}$}; 
\node (x2) at (2.5,6.0) [vertexYb] {$x_{1}$};
\node (x3) at (7.5,6.0) [vertexYb] {$x_{2}$}; 
\node (x4) at (7.5,1.0) [vertexYb] {$x_{3}$}; 
\node (y) at (5.0,3.5) [vertexYr] {$y$};
\draw [thick] (y) -- (x1); 
\draw [thick] (y) -- (x2); 
\draw [thick] (y) -- (x3); 
\draw [thick] (y) -- (x4); 
\end{tikzpicture} \hfill \hfill 
\caption{Capacitated graph $(G_1,\kkk_1)$ and its unique Nash subgraph.}\label{fig:4}
\end{figure}

We conclude this section by considering the capacitated graph $(G_2,\kkk_2)$ from Figure~\ref{fig:1}.
The auxiliary graph $G^{\aux}(G_2,\kkk_2)$ is simply the two isolated vertices $x$ and $z$.
In the second paragraph of the proof of Theorem~\ref{thm:P} it is noted that if $Z \neq \emptyset$ then by picking different edges incident with the vertices in $Z$ we obtain different $D$-sets. Thus, $(G_2,\kkk_2)$ has at least two Nash subgraphs.
In Figure~\ref{fig:5} below we depict $(G_2,\kkk_2)$ and its three Nash subgraphs, $H_1, H_2$, and $H_3$ (the colour coding of the vertices in the Nash subgraphs is the same as that used in Figure~\ref{fig:4}).
We note that all three Nash subgraphs of $(G_2,\kkk_2)$ have the same $D$-set $\set{x, z}$.

\begin{figure}[hbt!]
\centering
\tikzstyle{vertexX}=[circle,draw, top color=gray!10, bottom color=gray!70, minimum size=10pt, scale=0.9, inner sep=0.5pt]
\tikzstyle{vertexY}=[circle,draw, top color=black!10, bottom color=gray!70, minimum size=15pt, scale=0.8, inner sep=0.4pt]
\tikzstyle{vertexYb}=[circle,draw, top color=blue!40, bottom color=blue!40, minimum size=15pt, scale=0.8, inner sep=0.4pt]
\tikzstyle{vertexYr}=[circle,draw, top color=red!40, bottom color=red!40, minimum size=15pt, scale=0.8, inner sep=0.4pt]
\begin{tikzpicture}[scale=0.3]
\draw (5,-1.5) node {{\small $(G_2,\kkk_2)$}};
\draw (1.5,1.0) node {{\small $3$}};
\draw (8.5,1.0) node {{\small $1$}};
\draw (1.0,7.4) node {{\small $0$}};
\draw (5.0,7.4) node {{\small $0$}};
\draw (9.0,7.4) node {{\small $0$}};
\node (x) at (3.0,1.0) [vertexY] {$x$}; 
\node (z) at (7.0,1.0) [vertexY] {$z$}; 
\node (y1) at (1.0,6.0) [vertexY] {$y_{1}$}; 
\node (y2) at (5.0,6.0) [vertexY] {$y_{2}$}; 
\node (y3) at (9.0,6.0) [vertexY] {$y_{3}$}; 
\draw [thick] (x) -- (y1); 
\draw [thick] (x) -- (y2); 
\draw [thick] (x) -- (y3); 
\draw [thick] (z) -- (y1); 
\draw [thick] (z) -- (y2); 
\draw [thick] (z) -- (y3); 
\end{tikzpicture} \hfill \hfill
\begin{tikzpicture}[scale=0.3]
\draw (5,-1.5) node {{\small $H_1$}};
\node (x) at (3.0,1.0) [vertexYb] {$x$};
\node (z) at (7.0,1.0) [vertexYb] {$z$};
\node (y1) at (2.0,6.0) [vertexYr] {$y_{1}$};
\node (y2) at (5.0,6.0) [vertexYr] {$y_{2}$};
\node (y3) at (8.0,6.0) [vertexYr] {$y_{3}$};
\draw  (x) -- (y1);
\draw  (x) -- (y2);
\draw  (x) -- (y3);
\draw  (z) -- (y1);
\end{tikzpicture} \hfill
\begin{tikzpicture}[scale=0.3]
\draw (5,-1.5) node {{\small $H_2$}};
\node (x) at (3.0,1.0) [vertexYb] {$x$};
\node (z) at (7.0,1.0) [vertexYb] {$z$};
\node (y1) at (2.0,6.0) [vertexYr] {$y_{1}$};
\node (y2) at (5.0,6.0) [vertexYr] {$y_{2}$};
\node (y3) at (8.0,6.0) [vertexYr] {$y_{3}$};
\draw  (x) -- (y1);
\draw  (x) -- (y2);
\draw  (x) -- (y3);
\draw  (z) -- (y2);
\end{tikzpicture} \hfill
\begin{tikzpicture}[scale=0.3]
\draw (5,-1.5) node {{\small $H_3$}};
\node (x) at (3.0,1.0) [vertexYb] {$x$};
\node (z) at (7.0,1.0) [vertexYb] {$z$};
\node (y1) at (2.0,6.0) [vertexYr] {$y_{1}$};
\node (y2) at (5.0,6.0) [vertexYr] {$y_{2}$};
\node (y3) at (8.0,6.0) [vertexYr] {$y_{3}$};
\draw  (x) -- (y1);
\draw  (x) -- (y2);
\draw  (x) -- (y3);
\draw  (z) -- (y3);
\end{tikzpicture} \hfill
\caption{Capacitated graph $(G_2,\kkk_2)$ and its three Nash subgraphs $H_1, H_2$ and $H_3$.}\label{fig:5}
\end{figure}

\section{Complexity of Uniqueness of $D$-set}\label{sec:Ds}

In this section we consider the complexity of {\sc $D$-set Uniqueness}.
We will show that it is {\sf co-NP}-complete to decide whether a capacitated graph $(G,\kappa)$ has a unique $D$-set.
However, we will also see that there are some special cases of capacitated graphs for which {\sc $D$-set Uniqueness} can be decided efficiently.

If $Z=\emptyset$ then $(G,\kkk)$ has a unique $D$-set if and only if $(G,\kkk)$ has a unique Nash subgraph (this follows from the proof of Theorem \ref{thm:P}).
Thus, if $Z=\emptyset$ then by Theorem \ref{thm:P} it is polynomial to decide whether $G$ has a unique $D$-set.

However, as we can see below, in general, it is {\sf co-NP}-complete to decide whether a capacitated graph $(G,\kappa)$ has a unique $D$-set (the {\sc $D$-set Uniqueness} problem).
To refine this result, we consider the case when $\kappa(v)=k$ for every $v\in V(G).$ 
As we observed in Section \ref{sec:prel}, if $k=0$ then $V(G)$ is the only $D$-set in $G.$
The next theorem shows that {\sc $D$-set Uniqueness} remains in {\sf P} when $k=1$.
However, Theorem \ref{thm:NP} shows that
for $k\ge 2$, {\sc $D$-set Uniqueness} is  {\sf co-NP}-complete.

\begin{theorem} 
Let $(G,\kappa)$ be a capacitated graph and let $\kk{x}=1$ for all $x \in V(G)$.
Let $X = \{x \; | \; d_G(x)=1 \}$, $Y = N(X)\setminus X$ and 
$Z = V(G) \setminus (X \cup Y)$.
Then all \PDdecomps{} have the same $D$-set if and only if
$X \cup Z$ is independent and $|N_G(y) \cap X| \geq 2$ for all $y \in Y$.
In particular, {\sc $D$-set Uniqueness} is in {\sf P}
in this case.
\end{theorem}

\begin{proof}
If $X \cup Z$ is not independent then we are done by Lemma~\ref{lem_XZ_indep},
so assume that $X \cup Z$ is independent.
By Lemma~\ref{lem_XZ_indep}, there exists a \PDdecomp, $(D,P;E')$, such that $Y = P$.
If  $|N_G(y) \cap X| < 2$ for some $y \in Y$, then 
by Lemma~\ref{lem_vertex_u} there exists a \PDdecomp{}, $(D',P';E'')$, of $G,$ where
$y \in D'$. This implies that there exists \PDdecomps{} where $y$ belongs to its $D$-set 
and where $y$ belongs to its $P$-set.
Therefore we can conclude that $(G,\kappa)$ is a NO-instance for {\sc $D$-set Uniqueness}.

We now assume that $X \cup Z$ is independent and  $|N_G(y) \cap X| \geq 2$ for all $y \in Y$.
We will prove that all \PDdecomps{} have the same $D$-set in $(G,\kappa)$ and
we will do this by proving that  $M^{**}(G,\kkk{})$ holds, which by
Corollary~\ref{mainCOR} implies the desired result.

Recall that $M^{**}(G,\kkk{})$ holds if for every set $\emptyset \not= W \subseteq Y$ 
we have $|L(W)|>|W|$ (as $Y^{\kkk{}>0}=Y$ and $W^{\kkk{}}=W$). 
Let $W$ be arbitrary such that  $\emptyset \not= W \subseteq Y$.
By the definition of $L(W),$ we have that $|L(W)| \geq |N(W) \cap X|$.
As no vertex in $X$ has edges to more than one vertex in $Y$ (as $d_G(x)=1$) 
we have that $|N(W) \cap X| = \sum_{w \in W} |N(w) \cap X| \geq 2|W|$.
Therefore, we have 
\[
|L(W)| \geq |N(W) \cap X| \geq 2|W| > |W|.
\]
implying that $M^{**}(G,\kkk{})$ holds, as desired. 
\end{proof}

The following result is proved by reductions from $3$-SAT. This reduction is direct for the case of $k=2,$
where for an instance $I$ of $3$-SAT formula, we can construct a capacitated graph $(G,\kappa)$ such that 
$\kappa(x)=2$ for every vertex $x$ of $G$ and $(G,\kappa)$ contains at least two distinct $D$-sets if and only if $I$ is satisfiable.
In the case of $k\ge 3,$ we first trivially reduce from $3$-SAT to {\sc $k$-out-of-$(k+2)$-SAT}, where a CNF formula $F$ has $k+2$ 
literals in every clause and $F$ is satisfied if and only if there is a truth assignment which satisfies at least $k$ literals in every clause.
Then we reduce from {\sc $k$-out-of-$(k+2)$-SAT} to the complement of {\sc $D$-set Uniqueness}. While the main proof structure is similar in both cases,
the constructions of $(G,\kappa)$ are different. 

\begin{theorem}\label{thm:NP} Let $k\ge 2$ be an integer.
{\sc $D$-set Uniqueness} is {\sf co-NP}-complete for  capacitated graphs $(G,\kappa)$ with $\kk{x}=k$ for all $x \in V(G)$.
\end{theorem}


\begin{proof}

\noindent{\bf Case 1: $k=2$.}
We will reduce from $3$-SAT.
Let $I = C_1 \wedge C_2 \wedge \cdots \wedge C_m$ be an instance of $3$-SAT. 
Let $v_1,v_2,\ldots,v_n$ be the variables in $I$. Assume without loss of generality that $n$ is even (otherwise, we add to $I$ 
a new clause which contains three new variables).
We will now build a graph $G$ and let $\kappa(x)=2$ for every vertex $x$ of $G$, 
such that there exist \PDdecomps{} of $G$ with distinct $D$-sets if and only if $I$ is satisfiable.

First define $G_1$ as follows. 
Let $W_i = \{w_i,\bar{w}_i\}$ for all $i=1,2,\ldots,n$.
Let $V(G_1) = W_1 \cup W_2 \cup \cdots \cup W_n \cup \{r_1,r_2,\ldots,r_n\}$.
Let $E(G_1)$ contain all edges between $W_{2i-1}$ and $W_{2i}$ and the edge $r_{2i} r_{2i+1}$ for $i=1,2,\ldots,\frac{n}{2}$ (where $r_{n+1}=r_1$)
and edges between $r_i$ and $W_i$ for $i=1,2,\ldots,n$. The graph $G_1$ is illustrated below when $n=6$.

\2

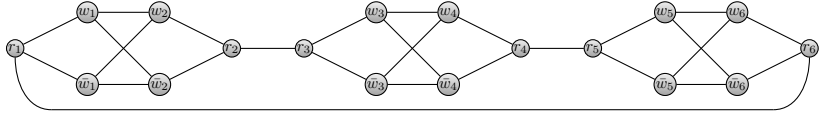
\begin{center}
\tikzstyle{vertexX}=[circle,draw, top color=gray!10, bottom color=gray!70, minimum size=10pt, scale=0.6, inner sep=0.1pt]
\tikzstyle{vertexY}=[circle,draw, top color=black!50, bottom color=gray!70, minimum size=8pt, scale=0.6, inner sep=0.1pt]
\begin{tikzpicture}[scale=0.48]
\node (w1) at (3.0,6.0) [vertexX] {$w_{1}$}; 
\node (w1b) at (3.0,4.0) [vertexX] {$\bar{w}_{1}$}; 
\node (w2) at (5.0,6.0) [vertexX] {$w_{2}$}; 
\node (w2b) at (5.0,4.0) [vertexX] {$\bar{w}_{2}$}; 
\node (r2) at (7.0,5.0) [vertexX] {$r_{2}$}; 
\node (r3) at (9.0,5.0) [vertexX] {$r_{3}$}; 
\node (w3) at (11.0,6.0) [vertexX] {$w_{3}$}; 
\node (w3b) at (11.0,4.0) [vertexX] {$\bar{w}_{3}$}; 
\node (w4) at (13.0,6.0) [vertexX] {$w_{4}$}; 
\node (w4b) at (13.0,4.0) [vertexX] {$\bar{w}_{4}$}; 
\node (r4) at (15.0,5.0) [vertexX] {$r_{4}$}; 
\node (r5) at (17.0,5.0) [vertexX] {$r_{5}$}; 
\node (w5) at (19.0,6.0) [vertexX] {$w_{5}$}; 
\node (w5b) at (19.0,4.0) [vertexX] {$\bar{w}_{5}$}; 
\node (w6) at (21.0,6.0) [vertexX] {$w_{6}$}; 
\node (w6b) at (21.0,4.0) [vertexX] {$\bar{w}_{6}$}; 
\node (r6) at (23.0,5.0) [vertexX] {$r_{6}$}; 
\node (r7) at (1.0,5.0) [vertexX] {$r_{1}$}; 
\draw  (w1) -- (w2); 
\draw  (w1) -- (w2b); 
\draw  (w1b) -- (w2); 
\draw  (w1b) -- (w2b); 
\draw  (w2) -- (r2); 
\draw  (w2b) -- (r2); 
\draw  (r2) -- (r3); 
\draw  (r3) -- (w3); 
\draw  (r3) -- (w3b); 
\draw  (w3) -- (w4); 
\draw  (w3) -- (w4b); 
\draw  (w3b) -- (w4); 
\draw  (w3b) -- (w4b); 
\draw  (w4) -- (r4); 
\draw  (w4b) -- (r4); 
\draw  (r4) -- (r5); 
\draw  (r5) -- (w5); 
\draw  (r5) -- (w5b); 
\draw  (w5) -- (w6); 
\draw  (w5) -- (w6b); 
\draw  (w5b) -- (w6); 
\draw  (w5b) -- (w6b); 
\draw  (w6) -- (r6); 
\draw  (w6b) -- (r6); 
\draw  (r7) -- (w1); 
\draw  (r7) -- (w1b); 
\draw  (r7) to [out=270, in=180] (2,3.3); 
 \draw  (2,3.3) -- (22.0,3.3); 
 \draw  (22.0,3.3) to [out=0, in=270] (r6); 
 
\end{tikzpicture}
\end{center}

\2

Let $G_1^s$ be the graph obtained from $G_1$ after subdividing every edge once. Let $u(e)$ denote the new vertex used to subdivide the edge $e \in E(G_1)$ and let
$U=\{u(e)| e\in E(G_1)\}.$
The graph $G_1^s$ is illustrated below when $n=6$.

\2

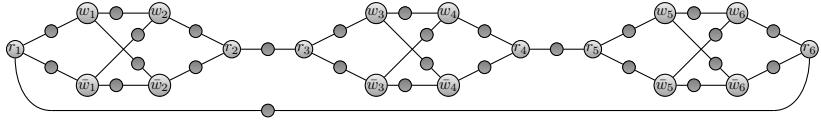
\begin{center}
\tikzstyle{vertexX}=[circle,draw, top color=gray!10, bottom color=gray!70, minimum size=10pt, scale=0.6, inner sep=0.1pt]
\tikzstyle{vertexY}=[circle,draw, top color=black!50, bottom color=gray!70, minimum size=8pt, scale=0.6, inner sep=0.1pt]
\tikzstyle{vertexZ}=[circle,draw, top color=gray!10, bottom color=gray!20, minimum size=14pt, scale=0.7, inner sep=0.1pt]
\begin{tikzpicture}[scale=0.48]
\node (w1) at (3.0,6.0) [vertexX] {$w_{1}$}; 
\node (w1b) at (3.0,4.0) [vertexX] {$\bar{w}_{1}$}; 
\node (w2) at (5.0,6.0) [vertexX] {$w_{2}$}; 
\node (w2b) at (5.0,4.0) [vertexX] {$\bar{w}_{2}$}; 
\node (r2) at (7.0,5.0) [vertexX] {$r_{2}$}; 
\node (r3) at (9.0,5.0) [vertexX] {$r_{3}$}; 
\node (w3) at (11.0,6.0) [vertexX] {$w_{3}$}; 
\node (w3b) at (11.0,4.0) [vertexX] {$\bar{w}_{3}$}; 
\node (w4) at (13.0,6.0) [vertexX] {$w_{4}$}; 
\node (w4b) at (13.0,4.0) [vertexX] {$\bar{w}_{4}$}; 
\node (r4) at (15.0,5.0) [vertexX] {$r_{4}$}; 
\node (r5) at (17.0,5.0) [vertexX] {$r_{5}$}; 
\node (w5) at (19.0,6.0) [vertexX] {$w_{5}$}; 
\node (w5b) at (19.0,4.0) [vertexX] {$\bar{w}_{5}$}; 
\node (w6) at (21.0,6.0) [vertexX] {$w_{6}$}; 
\node (w6b) at (21.0,4.0) [vertexX] {$\bar{w}_{6}$}; 
\node (r6) at (23.0,5.0) [vertexX] {$r_{6}$}; 
\node (r7) at (1.0,5.0) [vertexX] {$r_{1}$}; 
\node (subX) at (8.0,3.3) [vertexY] {}; 
\node (sub0) at (3.8,6.0) [vertexY] {}; 
\draw  (w1) -- (sub0); 
\draw  (sub0) -- (w2); 
\node (sub1) at (4.4,4.6) [vertexY] {}; 
\draw  (w1) -- (sub1); 
\draw  (sub1) -- (w2b); 
\node (sub2) at (4.4,5.3999999999999995) [vertexY] {}; 
\draw  (w1b) -- (sub2); 
\draw  (sub2) -- (w2); 
\node (sub3) at (3.8,4.0) [vertexY] {}; 
\draw  (w1b) -- (sub3); 
\draw  (sub3) -- (w2b); 
\node (sub4) at (6.0,5.5) [vertexY] {}; 
\draw  (w2) -- (sub4); 
\draw  (sub4) -- (r2); 
\node (sub5) at (6.0,4.5) [vertexY] {}; 
\draw  (w2b) -- (sub5); 
\draw  (sub5) -- (r2); 
\node (sub6) at (8.0,5.0) [vertexY] {}; 
\draw  (r2) -- (sub6); 
\draw  (sub6) -- (r3); 
\node (sub7) at (10.0,5.5) [vertexY] {}; 
\draw  (r3) -- (sub7); 
\draw  (sub7) -- (w3); 
\node (sub8) at (10.0,4.5) [vertexY] {}; 
\draw  (r3) -- (sub8); 
\draw  (sub8) -- (w3b); 
\node (sub9) at (11.8,6.0) [vertexY] {}; 
\draw  (w3) -- (sub9); 
\draw  (sub9) -- (w4); 
\node (sub10) at (12.399999999999999,4.6) [vertexY] {}; 
\draw  (w3) -- (sub10); 
\draw  (sub10) -- (w4b); 
\node (sub11) at (12.399999999999999,5.3999999999999995) [vertexY] {}; 
\draw  (w3b) -- (sub11); 
\draw  (sub11) -- (w4); 
\node (sub12) at (11.8,4.0) [vertexY] {}; 
\draw  (w3b) -- (sub12); 
\draw  (sub12) -- (w4b); 
\node (sub13) at (14.0,5.5) [vertexY] {}; 
\draw  (w4) -- (sub13); 
\draw  (sub13) -- (r4); 
\node (sub14) at (14.0,4.5) [vertexY] {}; 
\draw  (w4b) -- (sub14); 
\draw  (sub14) -- (r4); 
\node (sub15) at (16.0,5.0) [vertexY] {}; 
\draw  (r4) -- (sub15); 
\draw  (sub15) -- (r5); 
\node (sub16) at (18.0,5.5) [vertexY] {}; 
\draw  (r5) -- (sub16); 
\draw  (sub16) -- (w5); 
\node (sub17) at (18.0,4.5) [vertexY] {}; 
\draw  (r5) -- (sub17); 
\draw  (sub17) -- (w5b); 
\node (sub18) at (19.8,6.0) [vertexY] {}; 
\draw  (w5) -- (sub18); 
\draw  (sub18) -- (w6); 
\node (sub19) at (20.4,4.6) [vertexY] {}; 
\draw  (w5) -- (sub19); 
\draw  (sub19) -- (w6b); 
\node (sub20) at (20.4,5.3999999999999995) [vertexY] {}; 
\draw  (w5b) -- (sub20); 
\draw  (sub20) -- (w6); 
\node (sub21) at (19.8,4.0) [vertexY] {}; 
\draw  (w5b) -- (sub21); 
\draw  (sub21) -- (w6b); 
\node (sub22) at (22.0,5.5) [vertexY] {}; 
\draw  (w6) -- (sub22); 
\draw  (sub22) -- (r6); 
\node (sub23) at (22.0,4.5) [vertexY] {}; 
\draw  (w6b) -- (sub23); 
\draw  (sub23) -- (r6); 
\node (sub24) at (2.0,5.5) [vertexY] {}; 
\draw  (r7) -- (sub24); 
\draw  (sub24) -- (w1); 
\node (sub25) at (2.0,4.5) [vertexY] {}; 
\draw  (r7) -- (sub25); 
\draw  (sub25) -- (w1b); 
\draw  (r7) to [out=270, in=180] (2,3.3); 
 \draw  (2,3.3) -- (subX); 
 \draw  (subX) -- (22.0,3.3); 
 \draw  (22.0,3.3) to [out=0, in=270] (r6); 
 
\end{tikzpicture}
\end{center}

\2

Let $Q=\{q_1,q_2\}$, let $X^*=\{x_1^*,x_2^*,x_3^*,x_4^*,x_5^*\}$, and for every $i=1,2,\ldots,n,$ let $Z_i = \{z_i^1,z_i^2,z_i^3,z_i^4,z_i^5\}.$  
Let $G_2$ be obtained from $G_1^s$ by adding the vertices $Q \cup X^* \cup Z_1 \cup Z_2 \cup \cdots \cup Z_n$. Furthermore, add all edges 
from $W_i$ to $Z_i$, all edges from $Z_i$ to $q_1$ for $i=1,2,\ldots,n$, and all edges from $Q$ to $X^*$. The graph $G_2$ 
is illustrated below when $n=6$.

\2

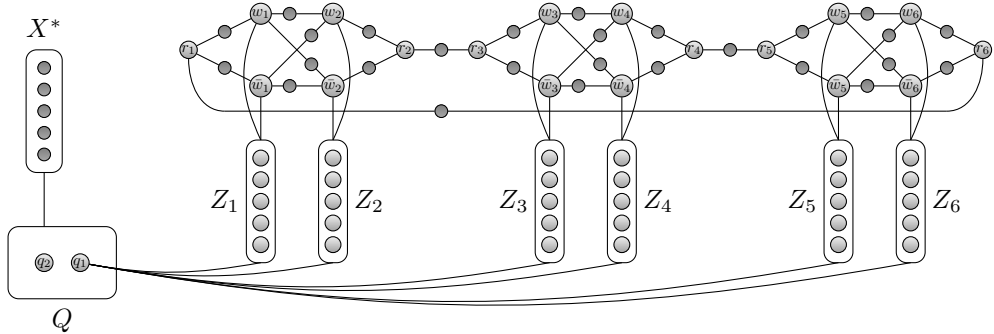
\begin{center}
\tikzstyle{vertexX}=[circle,draw, top color=gray!10, bottom color=gray!70, minimum size=10pt, scale=0.6, inner sep=0.1pt]
\tikzstyle{vertexY}=[circle,draw, top color=black!50, bottom color=gray!70, minimum size=8pt, scale=0.6, inner sep=0.1pt]
\tikzstyle{vertexZ}=[circle,draw, top color=gray!10, bottom color=gray!20, minimum size=14pt, scale=0.7, inner sep=0.1pt]
\begin{tikzpicture}[scale=0.48]
\node (w1) at (3.0,6.0) [vertexX] {$w_{1}$}; 
\node (w1b) at (3.0,4.0) [vertexX] {$\bar{w}_{1}$}; 
\node (w2) at (5.0,6.0) [vertexX] {$w_{2}$}; 
\node (w2b) at (5.0,4.0) [vertexX] {$\bar{w}_{2}$}; 
\node (r2) at (7.0,5.0) [vertexX] {$r_{2}$}; 
\node (r3) at (9.0,5.0) [vertexX] {$r_{3}$}; 
\node (w3) at (11.0,6.0) [vertexX] {$w_{3}$}; 
\node (w3b) at (11.0,4.0) [vertexX] {$\bar{w}_{3}$}; 
\node (w4) at (13.0,6.0) [vertexX] {$w_{4}$}; 
\node (w4b) at (13.0,4.0) [vertexX] {$\bar{w}_{4}$}; 
\node (r4) at (15.0,5.0) [vertexX] {$r_{4}$}; 
\node (r5) at (17.0,5.0) [vertexX] {$r_{5}$}; 
\node (w5) at (19.0,6.0) [vertexX] {$w_{5}$}; 
\node (w5b) at (19.0,4.0) [vertexX] {$\bar{w}_{5}$}; 
\node (w6) at (21.0,6.0) [vertexX] {$w_{6}$}; 
\node (w6b) at (21.0,4.0) [vertexX] {$\bar{w}_{6}$}; 
\node (r6) at (23.0,5.0) [vertexX] {$r_{6}$}; 
\node (r7) at (1.0,5.0) [vertexX] {$r_{1}$}; 
\node (z1x0) at (3.0,2.0) [vertexX] {}; 
\node (z1x1) at (3.0,1.4) [vertexX] {}; 
\node (z1x2) at (3.0,0.8) [vertexX] {}; 
\node (z1x3) at (3.0,0.20000000000000018) [vertexX] {}; 
\node (z1x4) at (3.0,-0.3999999999999999) [vertexX] {}; 
\node (z2x0) at (5.0,2.0) [vertexX] {}; 
\node (z2x1) at (5.0,1.4) [vertexX] {}; 
\node (z2x2) at (5.0,0.8) [vertexX] {}; 
\node (z2x3) at (5.0,0.20000000000000018) [vertexX] {}; 
\node (z2x4) at (5.0,-0.3999999999999999) [vertexX] {}; 
\node (z3x0) at (11.0,2.0) [vertexX] {}; 
\node (z3x1) at (11.0,1.4) [vertexX] {}; 
\node (z3x2) at (11.0,0.8) [vertexX] {}; 
\node (z3x3) at (11.0,0.20000000000000018) [vertexX] {}; 
\node (z3x4) at (11.0,-0.3999999999999999) [vertexX] {}; 
\node (z4x0) at (13.0,2.0) [vertexX] {}; 
\node (z4x1) at (13.0,1.4) [vertexX] {}; 
\node (z4x2) at (13.0,0.8) [vertexX] {}; 
\node (z4x3) at (13.0,0.20000000000000018) [vertexX] {}; 
\node (z4x4) at (13.0,-0.3999999999999999) [vertexX] {}; 
\node (z5x0) at (19.0,2.0) [vertexX] {}; 
\node (z5x1) at (19.0,1.4) [vertexX] {}; 
\node (z5x2) at (19.0,0.8) [vertexX] {}; 
\node (z5x3) at (19.0,0.20000000000000018) [vertexX] {}; 
\node (z5x4) at (19.0,-0.3999999999999999) [vertexX] {}; 
\node (z6x0) at (21.0,2.0) [vertexX] {}; 
\node (z6x1) at (21.0,1.4) [vertexX] {}; 
\node (z6x2) at (21.0,0.8) [vertexX] {}; 
\node (z6x3) at (21.0,0.20000000000000018) [vertexX] {}; 
\node (z6x4) at (21.0,-0.3999999999999999) [vertexX] {}; 
\node (q1) at (-2.0,-0.8999999999999999) [vertexX] {$q_1$}; 
\node (q2) at (-3.0,-0.8999999999999999) [vertexX] {$q_2$}; 
\node (xx0) at (-3.0,2.1) [vertexY] {}; 
\node (xx1) at (-3.0,2.7) [vertexY] {}; 
\node (xx2) at (-3.0,3.3) [vertexY] {}; 
\node (xx3) at (-3.0,3.9) [vertexY] {}; 
\node (xx4) at (-3.0,4.5) [vertexY] {}; 
\node (subX) at (8.0,3.3) [vertexY] {}; 
\node (sub0) at (3.8,6.0) [vertexY] {}; 
\draw  (w1) -- (sub0); 
\draw  (sub0) -- (w2); 
\node (sub1) at (4.4,4.6) [vertexY] {}; 
\draw  (w1) -- (sub1); 
\draw  (sub1) -- (w2b); 
\node (sub2) at (4.4,5.3999999999999995) [vertexY] {}; 
\draw  (w1b) -- (sub2); 
\draw  (sub2) -- (w2); 
\node (sub3) at (3.8,4.0) [vertexY] {}; 
\draw  (w1b) -- (sub3); 
\draw  (sub3) -- (w2b); 
\node (sub4) at (6.0,5.5) [vertexY] {}; 
\draw  (w2) -- (sub4); 
\draw  (sub4) -- (r2); 
\node (sub5) at (6.0,4.5) [vertexY] {}; 
\draw  (w2b) -- (sub5); 
\draw  (sub5) -- (r2); 
\node (sub6) at (8.0,5.0) [vertexY] {}; 
\draw  (r2) -- (sub6); 
\draw  (sub6) -- (r3); 
\node (sub7) at (10.0,5.5) [vertexY] {}; 
\draw  (r3) -- (sub7); 
\draw  (sub7) -- (w3); 
\node (sub8) at (10.0,4.5) [vertexY] {}; 
\draw  (r3) -- (sub8); 
\draw  (sub8) -- (w3b); 
\node (sub9) at (11.8,6.0) [vertexY] {}; 
\draw  (w3) -- (sub9); 
\draw  (sub9) -- (w4); 
\node (sub10) at (12.399999999999999,4.6) [vertexY] {}; 
\draw  (w3) -- (sub10); 
\draw  (sub10) -- (w4b); 
\node (sub11) at (12.399999999999999,5.3999999999999995) [vertexY] {}; 
\draw  (w3b) -- (sub11); 
\draw  (sub11) -- (w4); 
\node (sub12) at (11.8,4.0) [vertexY] {}; 
\draw  (w3b) -- (sub12); 
\draw  (sub12) -- (w4b); 
\node (sub13) at (14.0,5.5) [vertexY] {}; 
\draw  (w4) -- (sub13); 
\draw  (sub13) -- (r4); 
\node (sub14) at (14.0,4.5) [vertexY] {}; 
\draw  (w4b) -- (sub14); 
\draw  (sub14) -- (r4); 
\node (sub15) at (16.0,5.0) [vertexY] {}; 
\draw  (r4) -- (sub15); 
\draw  (sub15) -- (r5); 
\node (sub16) at (18.0,5.5) [vertexY] {}; 
\draw  (r5) -- (sub16); 
\draw  (sub16) -- (w5); 
\node (sub17) at (18.0,4.5) [vertexY] {}; 
\draw  (r5) -- (sub17); 
\draw  (sub17) -- (w5b); 
\node (sub18) at (19.8,6.0) [vertexY] {}; 
\draw  (w5) -- (sub18); 
\draw  (sub18) -- (w6); 
\node (sub19) at (20.4,4.6) [vertexY] {}; 
\draw  (w5) -- (sub19); 
\draw  (sub19) -- (w6b); 
\node (sub20) at (20.4,5.3999999999999995) [vertexY] {}; 
\draw  (w5b) -- (sub20); 
\draw  (sub20) -- (w6); 
\node (sub21) at (19.8,4.0) [vertexY] {}; 
\draw  (w5b) -- (sub21); 
\draw  (sub21) -- (w6b); 
\node (sub22) at (22.0,5.5) [vertexY] {}; 
\draw  (w6) -- (sub22); 
\draw  (sub22) -- (r6); 
\node (sub23) at (22.0,4.5) [vertexY] {}; 
\draw  (w6b) -- (sub23); 
\draw  (sub23) -- (r6); 
\node (sub24) at (2.0,5.5) [vertexY] {}; 
\draw  (r7) -- (sub24); 
\draw  (sub24) -- (w1); 
\node (sub25) at (2.0,4.5) [vertexY] {}; 
\draw  (r7) -- (sub25); 
\draw  (sub25) -- (w1b); 
\draw  (r7) to [out=270, in=180] (2,3.3); 
 \draw  (2,3.3) -- (subX); 
 \draw  (subX) -- (22.0,3.3); 
 \draw  (22.0,3.3) to [out=0, in=270] (r6); 
 \draw [rounded corners] (2.6,2.5) rectangle (3.4,-0.8999999999999999); 
\draw  (3.0,-0.8999999999999999) to [out=190, in=350] (q1); 
 \draw  (2.0,0.8) node {$Z_{1}$}; 
 \draw  (w1) to [out=245, in=115] (3.0,2.5); 
 \draw  (w1b) -- (3.0,2.5); 
 \draw [rounded corners] (4.6,2.5) rectangle (5.4,-0.8999999999999999); 
\draw  (5.0,-0.8999999999999999) to [out=190, in=350] (q1); 
 \draw  (6.0,0.8) node {$Z_{2}$}; 
 \draw  (w2) to [out=295, in=65] (5.0,2.5); 
 \draw  (w2b) -- (5.0,2.5); 
 \draw [rounded corners] (10.6,2.5) rectangle (11.4,-0.8999999999999999); 
\draw  (11.0,-0.8999999999999999) to [out=190, in=350] (q1); 
 \draw  (10.0,0.8) node {$Z_{3}$}; 
 \draw  (w3) to [out=245, in=115] (11.0,2.5); 
 \draw  (w3b) -- (11.0,2.5); 
 \draw [rounded corners] (12.6,2.5) rectangle (13.4,-0.8999999999999999); 
\draw  (13.0,-0.8999999999999999) to [out=190, in=350] (q1); 
 \draw  (14.0,0.8) node {$Z_{4}$}; 
 \draw  (w4) to [out=295, in=65] (13.0,2.5); 
 \draw  (w4b) -- (13.0,2.5); 
 \draw [rounded corners] (18.6,2.5) rectangle (19.4,-0.8999999999999999); 
\draw  (19.0,-0.8999999999999999) to [out=190, in=350] (q1); 
 \draw  (18.0,0.8) node {$Z_{5}$}; 
 \draw  (w5) to [out=245, in=115] (19.0,2.5); 
 \draw  (w5b) -- (19.0,2.5); 
 \draw [rounded corners] (20.6,2.5) rectangle (21.4,-0.8999999999999999); 
\draw  (21.0,-0.8999999999999999) to [out=190, in=350] (q1); 
 \draw  (22.0,0.8) node {$Z_{6}$}; 
 \draw  (w6) to [out=295, in=65] (21.0,2.5); 
 \draw  (w6b) -- (21.0,2.5); 
 \draw [rounded corners] (-4,-1.9) rectangle (-1,0.10000000000000009); 
\draw  (-2.5,-2.5) node {$Q$}; 
 \draw [rounded corners] (-3.5,1.6) rectangle (-2.5,5.0); 
\draw  (-3,1.6) -- (-3,0.10000000000000009); 
 \draw  (-3,5.6) node {$X^*$}; 
 
\end{tikzpicture}
\end{center}

\2

We now construct $G$ from $G_2$ by adding the vertices $\{c_1,c_2,\ldots,c_m\}$ and the following edge for each $j=1,2,\ldots,m$.
If the clause $C_j$ contains the literal $v_i$ then add an edge from $c_j$ to $w_i$ and if
 $C_j$ contains the literal $\bar{v}_i$ then add an edge from $c_j$ to $\bar{w}_i$. Finally, add an edge
from $c_j$ to $q_1$ for each $j=1,2,\ldots,m$. This completes the construction of $G$.  If $I = (\bar{v}_2 \vee v_3 \vee v_5) \wedge \cdots$ and $I$ contains six variables
then $G$ is illustrated below.

\2

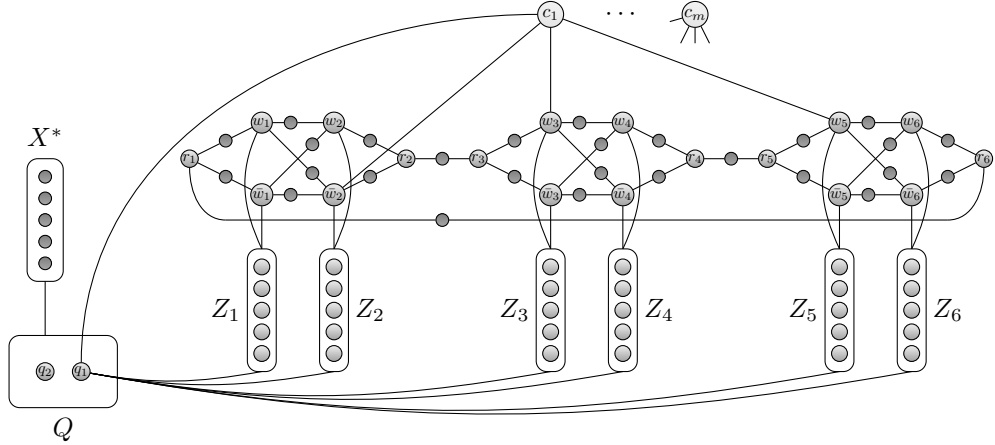
\begin{center}
\tikzstyle{vertexX}=[circle,draw, top color=gray!10, bottom color=gray!70, minimum size=10pt, scale=0.6, inner sep=0.1pt]
\tikzstyle{vertexY}=[circle,draw, top color=black!50, bottom color=gray!70, minimum size=8pt, scale=0.6, inner sep=0.1pt]
\tikzstyle{vertexZ}=[circle,draw, top color=gray!10, bottom color=gray!20, minimum size=14pt, scale=0.7, inner sep=0.1pt]
\begin{tikzpicture}[scale=0.48]
\node (w1) at (3.0,6.0) [vertexX] {$w_{1}$}; 
\node (w1b) at (3.0,4.0) [vertexX] {$\bar{w}_{1}$}; 
\node (w2) at (5.0,6.0) [vertexX] {$w_{2}$}; 
\node (w2b) at (5.0,4.0) [vertexX] {$\bar{w}_{2}$}; 
\node (r2) at (7.0,5.0) [vertexX] {$r_{2}$}; 
\node (r3) at (9.0,5.0) [vertexX] {$r_{3}$}; 
\node (w3) at (11.0,6.0) [vertexX] {$w_{3}$}; 
\node (w3b) at (11.0,4.0) [vertexX] {$\bar{w}_{3}$}; 
\node (w4) at (13.0,6.0) [vertexX] {$w_{4}$}; 
\node (w4b) at (13.0,4.0) [vertexX] {$\bar{w}_{4}$}; 
\node (r4) at (15.0,5.0) [vertexX] {$r_{4}$}; 
\node (r5) at (17.0,5.0) [vertexX] {$r_{5}$}; 
\node (w5) at (19.0,6.0) [vertexX] {$w_{5}$}; 
\node (w5b) at (19.0,4.0) [vertexX] {$\bar{w}_{5}$}; 
\node (w6) at (21.0,6.0) [vertexX] {$w_{6}$}; 
\node (w6b) at (21.0,4.0) [vertexX] {$\bar{w}_{6}$}; 
\node (r6) at (23.0,5.0) [vertexX] {$r_{6}$}; 
\node (r7) at (1.0,5.0) [vertexX] {$r_{1}$}; 
\node (z1x0) at (3.0,2.0) [vertexX] {}; 
\node (z1x1) at (3.0,1.4) [vertexX] {}; 
\node (z1x2) at (3.0,0.8) [vertexX] {}; 
\node (z1x3) at (3.0,0.20000000000000018) [vertexX] {}; 
\node (z1x4) at (3.0,-0.3999999999999999) [vertexX] {}; 
\node (z2x0) at (5.0,2.0) [vertexX] {}; 
\node (z2x1) at (5.0,1.4) [vertexX] {}; 
\node (z2x2) at (5.0,0.8) [vertexX] {}; 
\node (z2x3) at (5.0,0.20000000000000018) [vertexX] {}; 
\node (z2x4) at (5.0,-0.3999999999999999) [vertexX] {}; 
\node (z3x0) at (11.0,2.0) [vertexX] {}; 
\node (z3x1) at (11.0,1.4) [vertexX] {}; 
\node (z3x2) at (11.0,0.8) [vertexX] {}; 
\node (z3x3) at (11.0,0.20000000000000018) [vertexX] {}; 
\node (z3x4) at (11.0,-0.3999999999999999) [vertexX] {}; 
\node (z4x0) at (13.0,2.0) [vertexX] {}; 
\node (z4x1) at (13.0,1.4) [vertexX] {}; 
\node (z4x2) at (13.0,0.8) [vertexX] {}; 
\node (z4x3) at (13.0,0.20000000000000018) [vertexX] {}; 
\node (z4x4) at (13.0,-0.3999999999999999) [vertexX] {}; 
\node (z5x0) at (19.0,2.0) [vertexX] {}; 
\node (z5x1) at (19.0,1.4) [vertexX] {}; 
\node (z5x2) at (19.0,0.8) [vertexX] {}; 
\node (z5x3) at (19.0,0.20000000000000018) [vertexX] {}; 
\node (z5x4) at (19.0,-0.3999999999999999) [vertexX] {}; 
\node (z6x0) at (21.0,2.0) [vertexX] {}; 
\node (z6x1) at (21.0,1.4) [vertexX] {}; 
\node (z6x2) at (21.0,0.8) [vertexX] {}; 
\node (z6x3) at (21.0,0.20000000000000018) [vertexX] {}; 
\node (z6x4) at (21.0,-0.3999999999999999) [vertexX] {}; 
\node (q1) at (-2.0,-0.8999999999999999) [vertexX] {$q_1$}; 
\node (q2) at (-3.0,-0.8999999999999999) [vertexX] {$q_2$}; 
\node (xx0) at (-3.0,2.1) [vertexY] {}; 
\node (xx1) at (-3.0,2.7) [vertexY] {}; 
\node (xx2) at (-3.0,3.3) [vertexY] {}; 
\node (xx3) at (-3.0,3.9) [vertexY] {}; 
\node (xx4) at (-3.0,4.5) [vertexY] {}; 
\node (c1) at (11.0,9.0) [vertexZ] {$c_1$}; 
\node (clm) at (15.0,9.0) [vertexZ] {$c_m$}; 
\node (subX) at (8.0,3.3) [vertexY] {}; 
\node (sub0) at (3.8,6.0) [vertexY] {}; 
\draw  (w1) -- (sub0); 
\draw  (sub0) -- (w2); 
\node (sub1) at (4.4,4.6) [vertexY] {}; 
\draw  (w1) -- (sub1); 
\draw  (sub1) -- (w2b); 
\node (sub2) at (4.4,5.3999999999999995) [vertexY] {}; 
\draw  (w1b) -- (sub2); 
\draw  (sub2) -- (w2); 
\node (sub3) at (3.8,4.0) [vertexY] {}; 
\draw  (w1b) -- (sub3); 
\draw  (sub3) -- (w2b); 
\node (sub4) at (6.0,5.5) [vertexY] {}; 
\draw  (w2) -- (sub4); 
\draw  (sub4) -- (r2); 
\node (sub5) at (6.0,4.5) [vertexY] {}; 
\draw  (w2b) -- (sub5); 
\draw  (sub5) -- (r2); 
\node (sub6) at (8.0,5.0) [vertexY] {}; 
\draw  (r2) -- (sub6); 
\draw  (sub6) -- (r3); 
\node (sub7) at (10.0,5.5) [vertexY] {}; 
\draw  (r3) -- (sub7); 
\draw  (sub7) -- (w3); 
\node (sub8) at (10.0,4.5) [vertexY] {}; 
\draw  (r3) -- (sub8); 
\draw  (sub8) -- (w3b); 
\node (sub9) at (11.8,6.0) [vertexY] {}; 
\draw  (w3) -- (sub9); 
\draw  (sub9) -- (w4); 
\node (sub10) at (12.399999999999999,4.6) [vertexY] {}; 
\draw  (w3) -- (sub10); 
\draw  (sub10) -- (w4b); 
\node (sub11) at (12.399999999999999,5.3999999999999995) [vertexY] {}; 
\draw  (w3b) -- (sub11); 
\draw  (sub11) -- (w4); 
\node (sub12) at (11.8,4.0) [vertexY] {}; 
\draw  (w3b) -- (sub12); 
\draw  (sub12) -- (w4b); 
\node (sub13) at (14.0,5.5) [vertexY] {}; 
\draw  (w4) -- (sub13); 
\draw  (sub13) -- (r4); 
\node (sub14) at (14.0,4.5) [vertexY] {}; 
\draw  (w4b) -- (sub14); 
\draw  (sub14) -- (r4); 
\node (sub15) at (16.0,5.0) [vertexY] {}; 
\draw  (r4) -- (sub15); 
\draw  (sub15) -- (r5); 
\node (sub16) at (18.0,5.5) [vertexY] {}; 
\draw  (r5) -- (sub16); 
\draw  (sub16) -- (w5); 
\node (sub17) at (18.0,4.5) [vertexY] {}; 
\draw  (r5) -- (sub17); 
\draw  (sub17) -- (w5b); 
\node (sub18) at (19.8,6.0) [vertexY] {}; 
\draw  (w5) -- (sub18); 
\draw  (sub18) -- (w6); 
\node (sub19) at (20.4,4.6) [vertexY] {}; 
\draw  (w5) -- (sub19); 
\draw  (sub19) -- (w6b); 
\node (sub20) at (20.4,5.3999999999999995) [vertexY] {}; 
\draw  (w5b) -- (sub20); 
\draw  (sub20) -- (w6); 
\node (sub21) at (19.8,4.0) [vertexY] {}; 
\draw  (w5b) -- (sub21); 
\draw  (sub21) -- (w6b); 
\node (sub22) at (22.0,5.5) [vertexY] {}; 
\draw  (w6) -- (sub22); 
\draw  (sub22) -- (r6); 
\node (sub23) at (22.0,4.5) [vertexY] {}; 
\draw  (w6b) -- (sub23); 
\draw  (sub23) -- (r6); 
\node (sub24) at (2.0,5.5) [vertexY] {}; 
\draw  (r7) -- (sub24); 
\draw  (sub24) -- (w1); 
\node (sub25) at (2.0,4.5) [vertexY] {}; 
\draw  (r7) -- (sub25); 
\draw  (sub25) -- (w1b); 
\draw  (r7) to [out=270, in=180] (2,3.3); 
 \draw  (2,3.3) -- (subX); 
 \draw  (subX) -- (22.0,3.3); 
 \draw  (22.0,3.3) to [out=0, in=270] (r6); 
 \draw [rounded corners] (2.6,2.5) rectangle (3.4,-0.8999999999999999); 
\draw  (3.0,-0.8999999999999999) to [out=190, in=350] (q1); 
 \draw  (2.0,0.8) node {$Z_{1}$}; 
 \draw  (w1) to [out=245, in=115] (3.0,2.5); 
 \draw  (w1b) -- (3.0,2.5); 
 \draw [rounded corners] (4.6,2.5) rectangle (5.4,-0.8999999999999999); 
\draw  (5.0,-0.8999999999999999) to [out=190, in=350] (q1); 
 \draw  (6.0,0.8) node {$Z_{2}$}; 
 \draw  (w2) to [out=295, in=65] (5.0,2.5); 
 \draw  (w2b) -- (5.0,2.5); 
 \draw [rounded corners] (10.6,2.5) rectangle (11.4,-0.8999999999999999); 
\draw  (11.0,-0.8999999999999999) to [out=190, in=350] (q1); 
 \draw  (10.0,0.8) node {$Z_{3}$}; 
 \draw  (w3) to [out=245, in=115] (11.0,2.5); 
 \draw  (w3b) -- (11.0,2.5); 
 \draw [rounded corners] (12.6,2.5) rectangle (13.4,-0.8999999999999999); 
\draw  (13.0,-0.8999999999999999) to [out=190, in=350] (q1); 
 \draw  (14.0,0.8) node {$Z_{4}$}; 
 \draw  (w4) to [out=295, in=65] (13.0,2.5); 
 \draw  (w4b) -- (13.0,2.5); 
 \draw [rounded corners] (18.6,2.5) rectangle (19.4,-0.8999999999999999); 
\draw  (19.0,-0.8999999999999999) to [out=190, in=350] (q1); 
 \draw  (18.0,0.8) node {$Z_{5}$}; 
 \draw  (w5) to [out=245, in=115] (19.0,2.5); 
 \draw  (w5b) -- (19.0,2.5); 
 \draw [rounded corners] (20.6,2.5) rectangle (21.4,-0.8999999999999999); 
\draw  (21.0,-0.8999999999999999) to [out=190, in=350] (q1); 
 \draw  (22.0,0.8) node {$Z_{6}$}; 
 \draw  (w6) to [out=295, in=65] (21.0,2.5); 
 \draw  (w6b) -- (21.0,2.5); 
 \draw [rounded corners] (-4,-1.9) rectangle (-1,0.10000000000000009); 
\draw  (-2.5,-2.5) node {$Q$}; 
 \draw [rounded corners] (-3.5,1.6) rectangle (-2.5,5.0); 
\draw  (-3,1.6) -- (-3,0.10000000000000009); 
 \draw  (-3,5.6) node {$X^*$}; 
 \draw  (c1) -- (w2b); 
 \draw  (c1) -- (w3); 
 \draw  (c1) -- (w5); 
 \draw  (c1) to [out=180, in=90] (q1); 
 \draw  (13,9) node {$\cdots$}; 
 \draw  (clm) -- (14.6,8.2); 
 \draw  (clm) -- (15,8.2); 
 \draw  (clm) -- (15.4,8.2); 
 \draw  (clm) -- (14.3,8.8); 
 
\end{tikzpicture}
\end{center}

\2

We will now show that there exist \PDdecomps{} of $G$ with distinct $D$-sets if and only if $I$ is satisfiable.
We prove this using the following three claims. 

{\bf Claim A:} {\em There exists a \PDdecomp{} $(D,P;E')$ of $G$ such that $P=V(G_1) \cup Q$ and $D=V(G) \setminus P$.}

\2

{\em Proof of Claim A:}  This claim follows from Lemma \ref{lem_XZ_indep}. Indeed, $X(G,\kkk{})=X^*\cup U$, $Y(G,\kkk{})=Q\cup V(G_1)$
and $Z(G,\kkk{})=V(G)\setminus (X(G,\kkk{}) \cup Y(G,\kkk{})).$ Observe that $X(G,\kkk{})\cup Z(G,\kkk{})$ is an independent set.
Thus, $Y$ is a $P$-set in $G.$~$\diamond$

\2

{\bf Claim B:} {\em If $I$ is satisfiable then there exists a \PDdecomp{}, $(D,P;E')$, such that $P \not=V(G_1) \cup Q$.}

\2

{\em Proof of Claim B:} 
Assume that $I$ is satisfiable and let $\tau$ be a truth assignment to the variables in $I$ which satisfies $I.$
Construct $P$ and $D$ as follows. We start by constructing $P\setminus U$ and $D\setminus U.$
Since $D=V(G)\setminus P$, we will describe only $P\setminus U$. For every $i=1,2,\ldots,n$, 
if $v_i$ is true in $\tau$ then add the vertex $w_i$ to $P$ otherwise $\bar{w}_i$ to $P$.
Add $Q$ to $P.$ This completes construction of $P\setminus U.$


We will now distribute the vertices $u(e)$ in $P$ and $D$ for each $e \in E(G_1)$ and construct $E'$ such that $(D,P;E')$ is a 
\PDdecomp{} of $G$ with  $P \not=V(G_1) \cup Q$.
Let $E'$ contain all edges between $X^*$ and $Q$. 
For each $z \in Z_i$ add the edge $zq_1$ to $E'$ and add the edge between $z$ and the vertex in $W_j$ that belongs to $P$ to $E'$ for
$i=1,2,\ldots,n$. For each $j=1,2,\ldots,m$ add the edge $c_j q_1$ to $E'$ and the edge from $c_j$ to the vertex $w_i$ if 
$v_i$ is a true literal in $I$ and to the vertex $\bar{w}_i$ if $\bar{v}_j$ is a true literal in $I$ (just pick one true literal in $C_j$). 

Initially add all $u(e)$ to $P$ for $e \in V(G_1)$. We will move some of these vertices to $D$ below.
For each $i=1,2,\ldots,\frac{n}{2}$ proceed as follows.

\begin{itemize}
\item If $w_{2i-1}, w_{2i} \in D$: Move the vertex $u(\bar{w}_{2i-1}\bar{w}_{2i})$ to $D$.
Now add the edges shown in the below picture to $E'$.

 

\begin{center}
\tikzstyle{vertexX}=[circle,draw, top color=gray!10, bottom color=gray!70, minimum size=22pt, scale=0.5, inner sep=0.1pt]
\tikzstyle{vertexY}=[circle,draw, top color=black!50, bottom color=gray!70, minimum size=14pt, scale=0.6, inner sep=0.1pt]
\tikzstyle{vertexZ}=[circle,draw, top color=gray!10, bottom color=gray!20, minimum size=18pt, scale=0.7, inner sep=0.1pt]
\begin{tikzpicture}[scale=0.75]
\node (r1) at (1.0,5.0) [vertexX] {$r_{2i-1}$}; 
\node (w1) at (3.0,6.0) [vertexX] {$w_{2i-1}$}; 
\node (w1b) at (3.0,4.0) [vertexX] {$\bar{w}_{2i-1}$}; 
\node (w2) at (5.0,6.0) [vertexX] {$w_{2i}$}; 
\node (w2b) at (5.0,4.0) [vertexX] {$\bar{w}_{2i}$}; 
\node (r2) at (7.0,5.0) [vertexX] {$r_{2i}$}; 
\node (r3) at (9.0,5.0) [vertexX] {$r_{2i+1}$}; 
\node (sub0) at (2.0,5.5) [vertexY] {}; 
\draw  (r1) -- (sub0); 
\node (sub1) at (2.0,4.5) [vertexY] {}; 
\draw  (r1) -- (sub1); 
\node (sub2) at (3.8,6.0) [vertexY] {}; 
\draw  (w1) -- (sub2); 
\node (sub3) at (4.4,4.6) [vertexY] {}; 
\draw  (w1) -- (sub3); 
\node (sub4) at (4.4,5.3999999999999995) [vertexY] {}; 
\draw  (sub4) -- (w2); 
\node (sub5) at (3.8,4.0) [vertexY] {}; 
\draw  (w1b) -- (sub5); 
\draw  (sub5) -- (w2b); 
\node (sub6) at (6.0,5.5) [vertexY] {}; 
\draw  (w2) -- (sub6); 
\node (sub7) at (6.0,4.5) [vertexY] {}; 
\draw  (sub7) -- (r2); 
\node (sub8) at (8.0,5.0) [vertexY] {}; 
\draw  (r2) -- (sub8); 

\end{tikzpicture}
\end{center}
 

\item If $w_{2i-1},\bar{w}_{2i} \in D$: Move the vertex $u(\bar{w}_{2i-1}w_{2i})$ to $D$. 
Now add the edges shown in the below picture to $E'$.

 

\begin{center}
\tikzstyle{vertexX}=[circle,draw, top color=gray!10, bottom color=gray!70, minimum size=22pt, scale=0.5, inner sep=0.1pt]
\tikzstyle{vertexY}=[circle,draw, top color=black!50, bottom color=gray!70, minimum size=14pt, scale=0.6, inner sep=0.1pt]
\tikzstyle{vertexZ}=[circle,draw, top color=gray!10, bottom color=gray!20, minimum size=18pt, scale=0.7, inner sep=0.1pt]
\begin{tikzpicture}[scale=0.75]
\node (r1) at (1.0,5.0) [vertexX] {$r_{2i-1}$}; 
\node (w1) at (3.0,6.0) [vertexX] {$w_{2i-1}$}; 
\node (w1b) at (3.0,4.0) [vertexX] {$\bar{w}_{2i-1}$}; 
\node (w2) at (5.0,6.0) [vertexX] {$w_{2i}$}; 
\node (w2b) at (5.0,4.0) [vertexX] {$\bar{w}_{2i}$}; 
\node (r2) at (7.0,5.0) [vertexX] {$r_{2i}$}; 
\node (r3) at (9.0,5.0) [vertexX] {$r_{2i+1}$}; 
\node (sub0) at (2.0,5.5) [vertexY] {}; 
\draw  (r1) -- (sub0); 
\node (sub1) at (2.0,4.5) [vertexY] {}; 
\draw  (r1) -- (sub1); 
\node (sub2) at (3.8,6.0) [vertexY] {}; 
\draw  (w1) -- (sub2); 
\node (sub3) at (4.4,4.6) [vertexY] {}; 
\draw  (w1) -- (sub3); 
\node (sub4) at (4.4,5.3999999999999995) [vertexY] {}; 
\draw  (w1b) -- (sub4); 
\draw  (sub4) -- (w2); 
\node (sub5) at (3.8,4.0) [vertexY] {}; 
\draw  (sub5) -- (w2b); 
\node (sub6) at (6.0,5.5) [vertexY] {}; 
\draw  (sub6) -- (r2); 
\node (sub7) at (6.0,4.5) [vertexY] {}; 
\draw  (w2b) -- (sub7); 
\node (sub8) at (8.0,5.0) [vertexY] {}; 
\draw  (r2) -- (sub8); 

\end{tikzpicture}
\end{center}
 

\item If $\bar{w}_{2i-1}, w_{2i} \in D$: Move the vertex $u(w_{2i-1}\bar{w}_{2i})$ to $D$.
Now add the edges shown in the below picture to $E'$.

 

\begin{center}
\tikzstyle{vertexX}=[circle,draw, top color=gray!10, bottom color=gray!70, minimum size=22pt, scale=0.5, inner sep=0.1pt]
\tikzstyle{vertexY}=[circle,draw, top color=black!50, bottom color=gray!70, minimum size=14pt, scale=0.6, inner sep=0.1pt]
\tikzstyle{vertexZ}=[circle,draw, top color=gray!10, bottom color=gray!20, minimum size=18pt, scale=0.7, inner sep=0.1pt]
\begin{tikzpicture}[scale=0.75]
\node (r1) at (1.0,5.0) [vertexX] {$r_{2i-1}$}; 
\node (w1) at (3.0,6.0) [vertexX] {$w_{2i-1}$}; 
\node (w1b) at (3.0,4.0) [vertexX] {$\bar{w}_{2i-1}$}; 
\node (w2) at (5.0,6.0) [vertexX] {$w_{2i}$}; 
\node (w2b) at (5.0,4.0) [vertexX] {$\bar{w}_{2i}$}; 
\node (r2) at (7.0,5.0) [vertexX] {$r_{2i}$}; 
\node (r3) at (9.0,5.0) [vertexX] {$r_{2i+1}$}; 
\node (sub0) at (2.0,5.5) [vertexY] {}; 
\draw  (r1) -- (sub0); 
\node (sub1) at (2.0,4.5) [vertexY] {}; 
\draw  (r1) -- (sub1); 
\node (sub2) at (3.8,6.0) [vertexY] {}; 
\draw  (sub2) -- (w2); 
\node (sub3) at (4.4,4.6) [vertexY] {}; 
\draw  (w1) -- (sub3); 
\draw  (sub3) -- (w2b); 
\node (sub4) at (4.4,5.3999999999999995) [vertexY] {}; 
\draw  (w1b) -- (sub4); 
\node (sub5) at (3.8,4.0) [vertexY] {}; 
\draw  (w1b) -- (sub5); 
\node (sub6) at (6.0,5.5) [vertexY] {}; 
\draw  (w2) -- (sub6); 
\node (sub7) at (6.0,4.5) [vertexY] {}; 
\draw  (sub7) -- (r2); 
\node (sub8) at (8.0,5.0) [vertexY] {}; 
\draw  (r2) -- (sub8); 

\end{tikzpicture}
\end{center}
 

\item If $\bar{w}_{2i-1},\bar{w}_{2i} \in D$: Move the vertex $u(w_{2i-1}w_{2i})$ to $D$.
Now add the edges shown in the below picture to $E'$.

\begin{center}
\tikzstyle{vertexX}=[circle,draw, top color=gray!10, bottom color=gray!70, minimum size=22pt, scale=0.5, inner sep=0.1pt]
\tikzstyle{vertexY}=[circle,draw, top color=black!50, bottom color=gray!70, minimum size=14pt, scale=0.6, inner sep=0.1pt]
\tikzstyle{vertexZ}=[circle,draw, top color=gray!10, bottom color=gray!20, minimum size=18pt, scale=0.7, inner sep=0.1pt]
\begin{tikzpicture}[scale=0.75]
\node (r1) at (1.0,5.0) [vertexX] {$r_{2i-1}$}; 
\node (w1) at (3.0,6.0) [vertexX] {$w_{2i-1}$}; 
\node (w1b) at (3.0,4.0) [vertexX] {$\bar{w}_{2i-1}$}; 
\node (w2) at (5.0,6.0) [vertexX] {$w_{2i}$}; 
\node (w2b) at (5.0,4.0) [vertexX] {$\bar{w}_{2i}$}; 
\node (r2) at (7.0,5.0) [vertexX] {$r_{2i}$}; 
\node (r3) at (9.0,5.0) [vertexX] {$r_{2i+1}$}; 
\node (sub0) at (2.0,5.5) [vertexY] {}; 
\draw  (r1) -- (sub0); 
\node (sub1) at (2.0,4.5) [vertexY] {}; 
\draw  (r1) -- (sub1); 
\node (sub2) at (3.8,6.0) [vertexY] {}; 
\draw  (w1) -- (sub2); 
\draw  (sub2) -- (w2); 
\node (sub3) at (4.4,4.6) [vertexY] {}; 
\draw  (sub3) -- (w2b); 
\node (sub4) at (4.4,5.3999999999999995) [vertexY] {}; 
\draw  (w1b) -- (sub4); 
\node (sub5) at (3.8,4.0) [vertexY] {}; 
\draw  (w1b) -- (sub5); 
\node (sub6) at (6.0,5.5) [vertexY] {}; 
\draw  (sub6) -- (r2); 
\node (sub7) at (6.0,4.5) [vertexY] {}; 
\draw  (w2b) -- (sub7); 
\node (sub8) at (8.0,5.0) [vertexY] {}; 
\draw  (r2) -- (sub8); 

\end{tikzpicture}
\end{center}

\end{itemize}

After the above process we obtain the desired \PDdecomp{} $(D,P;E')$.~$\diamond$

\2

{\bf Claim C:} {\em If there exists a \PDdecomp{} $(D,P;E')$ such that $P \not=V(G_1) \cup Q$ then $I$ is satisfiable.}

\2

{\em Proof of Claim C:} Assume that there exists a \PDdecomp{} $(D,P;E')$ of $G$ such that $P \not= V(G_1) \cup Q$.
We now show the following subclaims which complete the proof of Claim~C and therefore of Case 1.

\2

{\bf Subclaim C.1:} {\em $X^* \subseteq D$ and $Q \subseteq P$.}

{\em Proof of Subclaim C.1:} Suppose that $Q \not\subseteq P$. This implies that each vertex in $X^*$ 
has at most one neighbour in $P$ and therefore $X^* \subseteq P$. However, then each vertex in $X^*$ has an edge to $D$ in $E'$, but these
five edges are all incident with $\{q_1,q_2\}$, a contradiction to $\kk{q_1}=\kk{q_2}=2$. Therefore, $Q \subseteq P$. 
This immediately implies that $X^* \subseteq D$, as if some $x \in X^*$ belonged to $P$ then it would have no edge to $D$ in $E'$. $\diamond$

\2

{\bf Subclaim C.2:} {\em $|D \cap W_i| \leq 1$ for all $i=1,2,3,\ldots,n$.}

{\em Proof of Subclaim C.2:}  Suppose that $|D \cap W_i| > 1$ for some $i \in \{1,2,\ldots,n\}$.
This implies that $\{w_i, \bar{w}_i \} \subseteq D$. In this case we must have $Z_i \subseteq P$ as every $z \in Z_i$ 
has exactly one neighbour in $P$ (namely $q_1$). By Subclaim~C.1 we note that $q_1 \in P$, which implies that every
vertex in $Z_i$ has an edge into $D$ in $E'$ and all these five edges must be incident with $W_i$.
This is a contradiction to $\kk{w_i}=\kk{\bar{w}_i}=2$ and thus to $|D \cap W_i| > 1.$ $\diamond$

\2

{\bf Subclaim C.3:} {\em $|D \cap W_i| =1$ for all $i=1,2,3,\ldots,n$ and $c_j \in D$ for all $j=1,2,\ldots,m$}

{\em Proof of Subclaim C.3:} Let $D' = D \cap V(G_1)$. If $D'=\emptyset$, then we note that $u(e) \in D$ for every $e \in E(G_1),$ $c_j\in D$ for every $j=1,2,\ldots,m$,
and $Z_i \subseteq D$ for every $i=1,2,\ldots,n$, by Subclaim~C.1. By Subclaim~C.1 we note that $P =V(G_1) \cup Q$, a contradiction to our 
assumption in the begining of the proof of Claim~C. Therefore $D' \not=\emptyset$.

Let $G'$ be the subgraph of $G_1$ induced by the vertices in $D'$. By Subclaim~C.2 we note that the maximum degree $\Delta(G')$ of $G'$  is at most 2, as 
$G_1$ is $3$-regular and every vertex in $G_1$ is adjacent to both vertices in $W_j$ for some $j$.
This implies that $|E(G')| \leq |V(G')| = |D'|$. Let $C_P = P \cap \{c_1,c_2,\ldots,c_m\}$. 

Suppose that $|E(G')| < |D'|+|C_P|$.
Let $E''$ denote all edges in $G_1$ that are incident with at least one vertex from $D'$.
As $G_1$ is $3$-regular we note that the following holds.
\[
|E''| = 3|D'| - |E(G')| > 3|D'| - (|D'|+|C_P|) = 2|D'|-|C_P|
\]
However, we note that $u(e) \in P$ for every $e \in E''$ (as it has degree two in $G$ and a neighbour in $D$).
These $|E''|$ edges as well as the $|C_P|$ edges from $C_P$ to $D$ in $E'$ are all incident
with $D'$ in $E'$, a contradiction to $|E''|+|C_P|>2|D'|$ (as $\kappa(s) =2$ for all vertices $s\in V(G)$).
Therefore $|E(G')| \geq |D'|+|C_P|$. As $|E(G')| \leq |D'|$ this implies that $|E(G')|=|D'|$ and $|C_P|=0$.

As $\Delta(G') \leq 2$ this implies that $G'$ is a collection of cycles (that is, $2$-regular).
By Subclaim~C.2 we note that the only possible cycle in $G'$ is the cycle of length $2n$ containing 
all $r_1,r_2,\ldots,r_n$ and exactly one vertex from each set $W_i$.  Therefore $D'$ contains
exactly one vertex from each $W_i$ for $i=1,2,\ldots,n$, as desired. 
As $|C_P|=0$ we also note that $c_j \in D$ for all $j=1,2,\ldots,m.$ $\diamond$

\2

{\bf Subclaim C.4:} {\em $I$ is satisfiable if we let $v_i$ be true when $w_i \in P$ and let $v_i$ be false when $w_i \in D$ for all $i=1,2,\ldots,n$.}

{\em Proof of Subclaim C.4:} For each $j=1,2,\ldots,m$ proceed as follows. By Subclaim~C.3 we note that $c_j \in D$ and therefore has two edges to $P$
in $E'$. At least one of these edges most go to a vertex in $V(G_1)$ as $c_j$ is only incident with one edge that is not incident with $V(G_1)$ 
(namely $c_j q_1$). If $c_j w_i \in E'$ then $w_i \in P$ and $v_i$ is true and $v_i$ is a literal of $C_j$, so $C_j$ is satisfied.
Alternatively if $c_j \bar{w}_i \in E'$ then $\bar{w}_i \in P$ (as $w_i \in D$ and Subclaim~C.3) and $v_i$ is false and $\bar{v}_i$ is a literal of $C_j$, 
so again $C_j$ is satisfied. This implies that $C_j$ is satisfied for all $j=1,2,\ldots,m$, which completes the proof of Case 1.

\vspace{3mm}

\noindent{\bf Case 2: $k\ge 3.$} We will reduce from {\sc $k$-out-of-$(k+2)$-SAT}.  That is, every clause will contain $k+2$ literals, and we need to satisfy at least $k$ of the 
literals for the clause to be satisfied.  This problem is {\sf NP}-complete, as we can reduce an instance of $3$-SAT to an instance of this problem by 
adding $k-1$ dummy variables, $v_1',v_2',\ldots,v_{k-1}'$ and adding these as literals to every clause.  

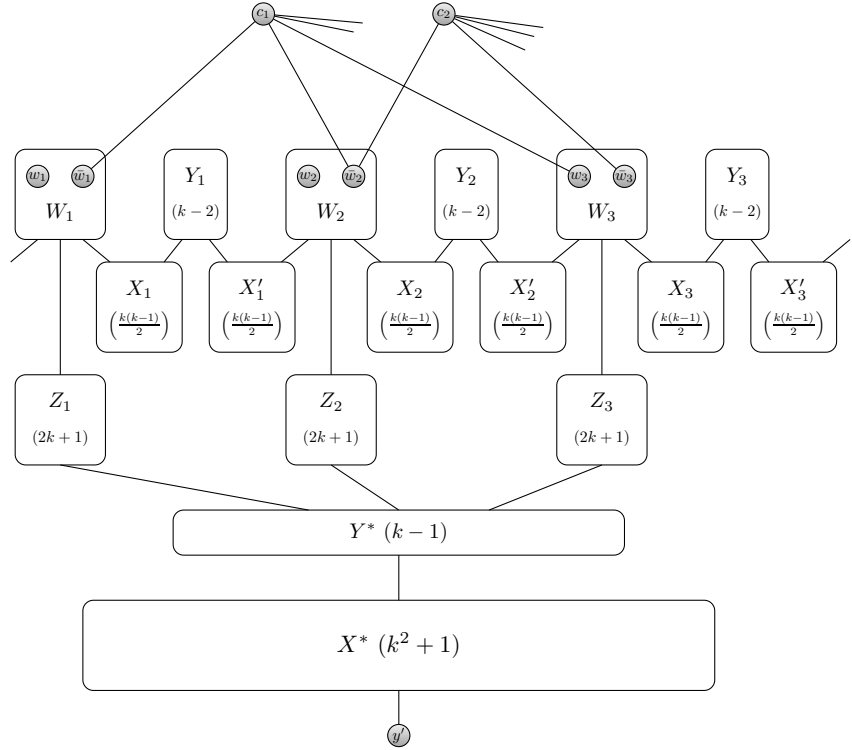
\begin{center}
\tikzstyle{vertexX}=[circle,draw, top color=gray!10, bottom color=gray!70, minimum size=14pt, scale=0.6, inner sep=0.1pt]
\begin{tikzpicture}[scale=0.6]

\draw [rounded corners] (2,1) rectangle (16,3);
\node [scale=0.9] at (9,2) {$X^*$ ($k^2+1$)};

\draw [rounded corners] (4,4) rectangle (14,5);
\node [scale=0.8] at (9,4.5) {$Y^*$ ($k-1$)};

\node (yp) at (9,0) [vertexX] {$y'$};

 \draw  (9,3) -- (9,4);       
 \draw  (yp) -- (9,1);        


\draw [rounded corners] (0.5,6) rectangle (2.5,8);
\node [scale=0.8] at (1.5,7.4) {$Z_1$};
\node [scale=0.6] at (1.5,6.6) {($2k+1$)};

 \node  (v1) at (1,12.4) [vertexX] {$w_{1}$};
 \node (vb1) at (2,12.4) [vertexX] {$\bar{w}_{1}$};
 \draw [rounded corners] (0.5,11) rectangle (2.5,13);
 \node [scale=0.8] at (1.5,11.6) {$W_1$};

 \draw  (1,11) -- (0.4,10.5);    

 \draw  (1.5,8) -- (1.5,11);     
 \draw  (1.5,6) -- (7,5);        
 \draw  (2,11) -- (2.6,10.5);    
 \draw  (3.8,10.5) -- (4.2,11);  
 \draw  (4.8,11) -- (5.2,10.5);  
 \draw  (7,11) -- (6.4,10.5);    

\draw [rounded corners] (2.3,8.5) rectangle (4.2,10.5);
\node [scale=0.8] at (3.25,9.9) {$X_1$};
\node [scale=0.6] at (3.25,9.1) {$\left(\frac{k(k-1)}{2}\right)$};

\draw [rounded corners] (3.8,11) rectangle (5.2,13);
\node [scale=0.8] at (4.5,12.4) {$Y_1$};
\node [scale=0.6] at (4.5,11.6) {($k-2$)};

\draw [rounded corners] (4.8,8.5) rectangle (6.7,10.5);
\node [scale=0.8] at (5.75,9.9) {$X_1'$};
\node [scale=0.6] at (5.75,9.1) {$\left(\frac{k(k-1)}{2}\right)$};


\draw [rounded corners] (6.5,6) rectangle (8.5,8);
\node [scale=0.8] at (7.5,7.4) {$Z_2$};
\node [scale=0.6] at (7.5,6.6) {($2k+1$)};

 \node  (v2) at (7,12.4) [vertexX] {$w_{2}$};
 \node (vb2) at (8,12.4) [vertexX] {$\bar{w}_{2}$};
 \draw [rounded corners] (6.5,11) rectangle (8.5,13);
 \node [scale=0.8] at (7.5,11.6) {$W_2$};

 \draw  (7.5,8) -- (7.5,11);     
 \draw  (7.5,6) -- (9,5);        
 \draw  (8,11) -- (8.6,10.5);    
 \draw  (9.8,10.5) -- (10.2,11);  
 \draw  (10.8,11) -- (11.2,10.5);  
 \draw  (13,11) -- (12.4,10.5);    

\draw [rounded corners] (8.3,8.5) rectangle (10.2,10.5);
\node [scale=0.8] at (9.25,9.9) {$X_2$};
\node [scale=0.6] at (9.25,9.1) {$\left(\frac{k(k-1)}{2}\right)$};

\draw [rounded corners] (9.8,11) rectangle (11.2,13);
\node [scale=0.8] at (10.5,12.4) {$Y_2$};
\node [scale=0.6] at (10.5,11.6) {($k-2$)};

\draw [rounded corners] (10.8,8.5) rectangle (12.7,10.5);
\node [scale=0.8] at (11.75,9.9) {$X_2'$};
\node [scale=0.6] at (11.75,9.1) {$\left(\frac{k(k-1)}{2}\right)$};


\draw [rounded corners] (12.5,6) rectangle (14.5,8);
\node [scale=0.8] at (13.5,7.4) {$Z_3$};
\node [scale=0.6] at (13.5,6.6) {($2k+1$)};

 \node  (v3) at (13,12.4) [vertexX] {$w_{3}$};
 \node (vb3) at (14,12.4) [vertexX] {$\bar{w}_{3}$};
 \draw [rounded corners] (12.5,11) rectangle (14.5,13);
 \node [scale=0.8] at (13.5,11.6) {$W_3$};

 \draw  (13.5,8) -- (13.5,11);     
 \draw  (13.5,6) -- (11,5);        
 \draw  (14,11) -- (14.6,10.5);    
 \draw  (15.8,10.5) -- (16.2,11);  
 \draw  (16.8,11) -- (17.2,10.5);  
 \draw  (19,11) -- (18.4,10.5);    

\draw [rounded corners] (14.3,8.5) rectangle (16.2,10.5);
\node [scale=0.8] at (15.25,9.9) {$X_3$};
\node [scale=0.6] at (15.25,9.1) {$\left(\frac{k(k-1)}{2}\right)$}; 

\draw [rounded corners] (15.8,11) rectangle (17.2,13);
\node [scale=0.8] at (16.5,12.4) {$Y_3$};
\node [scale=0.6] at (16.5,11.6) {($k-2$)};

\draw [rounded corners] (16.8,8.5) rectangle (18.7,10.5);
\node [scale=0.8] at (17.75,9.9) {$X_3'$};
\node [scale=0.6] at (17.75,9.1) {$\left(\frac{k(k-1)}{2}\right)$};


\node (c1) at (6,16) [vertexX] {$c_1$};

\draw  (c1) -- (vb1);    
\draw  (c1) -- (vb2);    
\draw  (c1) -- (v3);    
\draw  (c1) -- (8.2,15.8);
\draw  (c1) -- (8,15.6);

\node (c2) at (10,16) [vertexX] {$c_2$};

\draw  (c2) -- (vb2);   
\draw  (c2) -- (vb3);
\draw  (c2) -- (12.2,15.7);
\draw  (c2) -- (12,15.5);
\draw  (c2) -- (11.8,15.2);

\end{tikzpicture} 
\end{center}

Let $I = C_1 \wedge C_2 \wedge \cdots \wedge C_m$ be an instance of {\sc $k$-out-of-$(k+2)$-SAT} where each clause $C_i$ contains $k+2$ literals.
Let $v_1,v_2,\ldots,v_n$ be the variables in $I$. We will now build a graph $G$, where $\kappa(x)=k$ for every vertex $x$ of $G$.
For each $i=1,2,3,\ldots,n,$ let  $W_i=\{w_i,\bar{w}_i\}$, let $X_i$ and $X_i'$ be vertex sets of size $\frac{k(k-1)}{2},$ 
let $Y_i$ be a vertex set of size $k-2$, and let $Z_i$ be a vertex set of size $2k+1.$
Furthermore, let $X^*$ be a vertex set of size $k^2+1$, let $Y^*$  be a vertex set of size $k-1,$ and $C=\{c_1,c_2,\ldots,c_m\}.$ 
Let $y'$ be a vertex and define a graph $G$ as follows (see the illustration above). 
$$
V(G) = \{y'\} \cup Y^* \cup X^* \cup C \cup \left( \cup_{i=1}^n ( W_i \cup X_i \cup X_i' \cup Y_i \cup Z_i ) \right)
$$
We first add the following edges to $G$ for all $i=1,2,\ldots,n$ (where subscript $n+1$ is equivalent to $1$):
All edges from $Y^*$ to $Z_i$, from $Z_i$ to $W_i$,  from $W_i$ to $X_i$, from $X_i$ to $Y_i$,  from $Y_i$ to $X_i'$,  and from $X_i'$ to $W_{i+1}$.
Furthermore, add all edges from $y'$ to $X^*$ and all edges from $X^*$ to $Y^*$.
We then for all $j=1,2,\ldots,m$ add edges from $c_j$ to the vertices corresponding to the literals in the clause.  
(In the figure above, $C_1$ contains literals $\bar{v}_1,\bar{v}_2,v_3,\ldots$ and 
$C_2$ contains literals $\bar{v}_2,\bar{v}_3,\ldots$.)

Recall that $\kk{v} = k$ for all $v \in V(G)$. Define $X$, $Y$ and $Z$ as follows.

\2

Let $X = X^* \cup \left( \cup_{i=1}^n (X_i \cup X_i') \right)$.

Let $Y = N(X) - X = \{y'\} \cup Y^*  \cup \left( \cup_{i=1}^n (W_i \cup Y_i) \right)$.

Let $Z = V(G) \setminus (X \cup Y) = C \cup \left( \cup_{i=1}^n Z_i \right)$.

\2

We will now show that there exist Nash subgraphs of $G$ with at least two distinct $D$-sets if and only if $I$ is satisfiable.
We prove this using the following claims. 

\2

{\bf Claim A:} {\em There exists a Nash subgraph $(D,P;E')$ of $G$ such that $P=Y$ and $D=X \cup Z$.}

\2

{\em Proof of Claim A:}  This claim follows from Lemma \ref{lem_XZ_indep}. Indeed, $X(G,\kkk{})=X$, $Y(G,\kkk{})=Y$
and $Z(G,\kkk{})=Z.$ Observe that $X \cup Z$ is an independent set.
Thus, $Y$ is a $P$-set in $G.$

\2

{\bf Claim B:} {\em If $I$ is satisfiable then there exists a Nash subgraph $(D,P;E')$ such that $P \not=Y$.}

\2

{\em Proof of Claim B:}
Assume that $I$ is satisfiable and let $\tau$ be a truth assignment to the variables in $I$ which satisfies $I.$
Construct $P$ and $D$ as follows. 

\begin{itemize}
\item If $\tau(v_i)=$ {\sf true} then let $w_i \in P$ and $\bar{w}_i \in D$. 
\item  If $\tau(v_i)=$ {\sf false} then let $\bar{w}_i \in P$ and $w_i \in D$.
\item Let $X_i \cup X_i'$ belong to $P$ for all $i=1,2,\ldots,n$.
\item Let $Z_i \cup Y_i$ belong to $D$ for all $i=1,2,\ldots,n$.
\end{itemize}

Let $\{y'\} \cup Y^*$ belong to $P$ and $C \cup X^*$ to $D$. This defines $P$ and $D$. We now define $E'$.
Let $E'$ contain all edges between $X^*$ and $\{y'\} \cup Y^*$. 
For every vertex in $Z_i$ ($i=1,2,\ldots,n$) add an edge to the vertex in $W_i$ that belongs 
to $P$ and add all edges to $Y^*$. For each vertex in $W_i$ ($i=1,2,\ldots,n$) that belongs to $D$ add $k$ edges to 
$X_i$.  For each $i=1,2,\ldots,n$ add $k(k-1)/2 -k$ edges from $Y_i$ to $X_i$ and $k(k-1)/2$ edges from $Y_i$ to $X_i'$ in such a way 
that each of the $k-2$ vertices in $Y_i$ is incident with $k$ edges and each vertex in $X_i$ and $X_i'$ is incident with exactly one edge (from $W_i \cup Y_i$).
Finally, for each $j =1,2,\ldots,m$ pick $k$ true literals in $C_j$ and add an edge from $c_j$ to $w_i$ if $v_i$ is one of the true literals and add
an edge from $c_j$ to $\bar{w}_i$ if $\bar{v}_i$ is a true literal.  This completes the construction of $E'$. 
We note that $(D,P;E')$ is a Nash subgraph with $P \not=Y$.

\2

{\bf Claim C:} {\em If there exists a Nash subgraph $(D,P;E')$ such that $P \not=Y$ then $I$ is satisfiable.}

\2

{\em Proof of Claim C:} Assume that there exists a Nash subgraph $(D,P;E')$ of $G$ such that $P \not= Y$.
We now show the following subclaims which complete the proof of Claim~C and therefore of the theorem.

\2

{\bf Subclaim C.1:} {\em $X^* \subseteq D$ and $\{y'\} \cup Y^* \subseteq P$.}

{\em Proof of Subclaim C.1:}
Suppose that $\{y'\} \cup Y^* \not\subseteq P$. This implies that each vertex in $X^*$
has at most $k-1$ neighbours in $P$ and therefore $X^* \subseteq P$. However, then each vertex in $X^*$ has an edge to $D$ in $E'$, but these
$k^2+1$ edges are all incident with the $k$ vertices $\{y'\} \cup Y^*$, a contradiction to $\kappa(x)=k$ for every $x\in X^*$. Therefore, $\{y'\} \cup Y^* \subseteq P$.
This immediately implies that $X^* \subseteq D$, as if some $x \in X^*$ belonged to $P$ then it would have no edge to $D$ in $E'$. $\diamond$

\2

{\bf Subclaim C.2:} {\em $|D \cap W_i| \leq 1$ for all $i=1,2,3,\ldots,n$.}

{\em Proof of Subclaim C.2:}  Suppose that $|D \cap W_i| > 1$ for some $i \in \{1,2,\ldots,n\}$.
This implies that $\{w_i, \bar{w}_i \} \subseteq D$. In this case we must have $Z_i \subseteq P$ as every $z \in Z_i$
has at most $k-1$ neighbours in $P$ (namely the vertices in $Y^*$). By Subclaim~C.1 we note that $Y^* \subseteq P$, which implies that every
vertex in $Z_i$ has an edge into $D$ in $E'$ and all these $2k+1$ edges must be incident with $W_i$.
This is a contradiction to $\kk{w_i}=\kk{\bar{w}_i}=k$ and thus to $|D \cap W_i| > 1.$ $\diamond$

\2

{\bf Subclaim C.3:} {\em $|D \cap W_i| =1$ for all $i=1,2,\ldots,n$ and $c_j \in D$ for all $j=1,2,\ldots,m$}

{\em Proof of Subclaim C.3:} By Subclaim~C.2 we note that $|D \cap W_i| \leq 1$. Suppose
that $|D \cap W_i| =0$ for some $i \in \{1,2,\ldots,n\}$. We consider the following two cases.

\2

{\bf Case C.3.1: $|D \cap W_i| =0$ for all $i=1,2,\ldots,n$.}  If $D \cap Y_i \not= \emptyset$ for some $i$, then 
$X_i \cup X_i' \subseteq P$, as every vertex, $x$, in this set has $d_G(x)=\kk{x}=k$ and at least one neighbour is in $D$.
However the $k(k-1)$ vertices in $X_i \cup X_i'$ will all have edges into $D \cap Y_i$ (as $|D \cap W_i| =0$ for all $i=1,2,\ldots,n$)
in $E'$, a contradiction to $D \cap Y_i$ being incident with at most $k|Y_i| = k(k-2)$ edges. Therefore we may assume that
 $D \cap Y_i = \emptyset$ for all $i=1,2,\ldots,n$.

By Subclaim~C.1 we note that $Y = \{y'\} \cup Y^* \cup\left(\cup_{i=1}^n (W_i \cup Y_i)\right) \subseteq P$.
In this case we must have $Y=P$, as $G-Y$ is independent. This is a contradiction to $P \not= Y$.
This completes Case C.3.1.

\2

{\bf Case C.3.2: $|D \cap W_i| \not= 0$ for some $i \in \{1,2,\ldots,n\}$.}
Without loss of generality assume that $|D \cap W_1| \not= 0$. 
This implies that $X_1 \cup X_n \subseteq P$, as every vertex, $x \in X_1 \cup X_n$, has $d_G(x)=\kk{x}=k$ and at least one neighbour is in $D$.
Without loss of generality assume that the vertex in $D \cap W_1$ has at least as many edges to $X_n$ as to $X_1$ in $E'$. 
This implies that there are at most $k/2$ edges from $D \cap W_1$ to $X_1$ in $E'$. As $k/2 < k(k-1)/2$ we note that $Y_1 \cap D \not= \emptyset$.
This implies that $X_1' \subseteq P$, as every vertex, $x \in X_1'$, has $d_G(x)=\kk{x}=k$ and at least one neighbour is in $D$.
As $X_1 \cup X_1' \subseteq P$, $X_1 \cup X_1'$ has at least $|X_1|+|X_1'| = k(k-1)$ edges into $D \cap (W_1 \cup Y_1 \cup W_2)$.
This implies that $W_2 \cap D \not= \emptyset$, as there can be at most $k/2 + k(k-2)$ edges from $D \cap (W_1 \cup Y_1)$ to 
$X_1 \cup X_1'$ in $E'$. Furthermore there are at least $k/2$ edges from $D \cap W_2$ to $X_1'$ in $E'$.
Continueeing the above process we note that $W_3 \cap D \not= \emptyset,$ $W_4 \cap D \not= \emptyset$, etc.
This contradicts the fact that  $|D \cap W_i| = 0$ for some $i$. 
This completes Case C.3.2.

\2

By the above two cases we note that $|D \cap W_i| =1$ for all $i=1,2,\ldots,n$. This implies that 
$X_i \cup X_i' \subseteq P$ for all  $i=1,2,\ldots,n$.
Therefore there are at least $nk(k-1)$ edges from $\cup_{i=1}^n (X_i \cup X_i')$ to $D \cap \left(\cup_{i=1}^n (W_i \cup Y_i) \right)$.
As $|D \cap W_i| =1$ for all $i=1,2,\ldots,n$ there can be at most $nk + nk(k-2)$ such edges. This implies that there are exactly
$nk(k-1)$ edges from $\cup_{i=1}^n (X_i \cup X_i')$ to $D \cap \left(\cup_{i=1}^n (W_i \cup Y_i) \right)$ and there are no 
edges from $D \cap \left(\cup_{i=1}^n (W_i \cup Y_i) \right)$ to any vertex in $P$ that does not lie in $\cup_{i=1}^n (X_i \cup X_i')$.
This implies that  $c_j \in D$ for all $j=1,2,\ldots,m$, as if some $c_j \in P$ then it would not have any edge to $D$.$\diamond$

\2

{\bf Subclaim C.4:} {\em $I$ is satisfiable if we let $v_i$ be true when $w_i \in P$ and let $v_i$ be false when $w_i \in D$ for all $i=1,2,\ldots,n$.}

{\em Proof of Subclaim C.4:} For each $j=1,2,\ldots,m$ proceed as follows. By Subclaim~C.3 we note that $c_j \in D$ and therefore has $k$ edges to $P$
in $E'$. 
If $c_j w_i \in E'$ then $w_i \in P$ and $v_i$ is true and $v_i$ is a literal of $C_j$ which is satisfied.
Alternatively if $c_j \bar{w}_i \in E'$ then $\bar{w}_i \in P$ (as $w_i \in D$ and Subclaim~C.3) and $v_i$ is false and $\bar{v}_i$ is a literal of $C_j$
which again is satisfied.
As there are $k$ edges from $c_j$ to $P$ in $E'$ we obtain $k$ true literals in $C_j$, 
which implies that $C_j$ is satisfied for all $j=1,2,\ldots,m$, which completes the proof of the theorem.
\end{proof}

While we have proved that {\sc $D$-set Uniqueness} is {\sf co-NP}-complete, it is not hard to solve this problem in time ${\cal O}^*(2^n)$, where ${\cal O}^*$ hides not only coefficients, but also polynomials in $n.$
Indeed, we can consider every non-empty subset $S$ of $V(G)$ in turn and check whether $S$ is the $D$-set of a Nash subgraph of $(G,\kappa)$ using network flows. 
Conditional on the well-known Exponential Time Hypothesis (ETH)\footnote{ETH introduced by \cite{ImpagliazzoP01} conjectures that there is a constant $\delta>0$ such that 3-SAT cannot be solved in time ${\cal O}^*(2^{n \delta})$.} holding, one can show the following:

\begin{corollary}
Unless ETH fails, there is constant $\delta'>0$ such that  {\sc $D$-set Uniqueness } cannot be solved in time ${\cal O}^*(2^{n \delta'})$.
\end{corollary}
\begin{proof}
Let us consider the reduction from 3-SAT to {\sc  $D$-set Uniqueness}  in the proof of Theorem \ref{thm:NP} (Case 1, i.e. $k=2$). Let $I$ be an instance of 3-SAT with $n$ variables and $m$ clauses and let $p$ be the number of vertices in the graph $G$ we construct in the reduction. 
By the sparsification lemma of \cite{ImpagliazzoP01}, there is a constant $\delta'' > 0$ such that $I$ cannot be solved in time ${\cal O}^*(2^{n \delta''})$. It is not hard to see that $p=\Theta(n)+m.$ This implies our result.
\end{proof}


\section{Conclusion}\label{sec:conc}

In this paper, we have considered a variant of the dominating set problem that arises naturally in the study of public goods on networks when there are capacity constraints on sharing \citep{GerkeGHN}.
This problem, the exact capacitated dominating set problem, is captured by Nash subgraphs and their associated partite sets the $D$-set and the $P$-set.
Interestingly, and somewhat unusually for variants of the dominating set problem, the exact capacitated domination set problem admits unique solutions, and this has been our focus in this paper.

To be more specific, we have focused on two questions, {\sc Nash Subgraph Uniqueness} - {decide whether a capacitated graph has a unique Nash subgraph} - and {\sc $D$-set Uniqueness} - {decide whether a capacitated graph has a unique $D$-set}, and we considered their time complexities.
Despite the obvious similarities between the two problems, we show that their time complexities are very different (unless {\sf P}={\sf co-NP}).
That is, we show that {\sc Nash Subgraph Uniqueness} is polynomial-time solvable whereas {\sc $D$-set Uniqueness} is {\sf co-NP}-complete.



Consider a capacitated graph $(K_{3n},\kappa),$ where $n\ge 1$ and $\kappa(v)=2$ for every $v\in V(K_{3n}).$
(The capacitated graph $(G_3,\kkk_3)$ in Figure~\ref{fig:1} is one such an example where $n = 2$.)
Observe that every $p$-size subset of $V(K_{3n})$ for $n\le p\le 3n-2$ is a $D$-set.
Thus, a capacitated graph can have an exponential number of $D$-sets and hence of Nash subgraphs. This leads to the following open questions: (a) What is the complexity of counting all Nash subgraphs of a capacitated graph? (b) Is there an ${\cal O}^*({\rm dp}(G,\kappa))$-time algorithm to generate all Nash subgraphs of $(G,\kappa)$, where ${\rm dp}(G,\kappa)$ is the number of Nash subgraphs in $(G,\kappa)$?

\paragraph{Acknowledgement} Anders Yeo's research was partially supported by grant DFF-7014-00037B of Independent Research Fund Denmark.


\bibliographystyle{plainnat}
\bibliography{DsetRefs.bib}

\newpage

\appendix

\section*{APPENDIX}\label{APP}

\section{Proof of Theorem \ref{all_PD}}\label{sec:AppExistence}

The  proof  is a slight modification of the proof in \cite{GerkeGHN}.

Recall that by Assumption 1, for all $v \in V$ we have $\kk{v} \leq d(v)$.
The proof proceeds by induction on the number $n$ of vertices of $G.$ If $n=1$, then $G$ is a Nash subgraph with  $D=V(G)$ and $P=\emptyset$.
Now assume the claim is true for all graphs with fewer than $n\geq 2$ vertices, and let $G$ be a graph on $n$ vertices.

\2

\noindent{\bf Case 1: There is a vertex $u$ of degree equal $\kk{u}.$} Let $B$ be the star with center $\{u\}$ and
leaves $N_G(u)$ and let $G'=G-B$. 
If $G'$ has no vertices then $B$ is clearly a Nash subgraph of $G$ with $D$-set $\{u\}$ and $P$-set $N_G(u).$ 
Otherwise, by induction hypothesis, $G'$ has a Nash subgraph $H'=(D',P';E)$. 
Construct a subgraph $H$ of $G$ from the disjoint union of $H'$ and $B$ as follows. 
For every $v\in D'$ with $d_{G'}(v)<\kk{v},$ add exactly $\kappa(v)-d_{G'}(v)$ edges of $G$ between $v$ and $N_G(u).$ Set $D=D'\cup \{u\} $ and $P=P'\cup N(u).$
To see that $H$ is a Nash subgraph of $G,$ observe that (a) $H$ is a spanning bipartite subgraph of $G$ as $H'$ and $B$ are bipartite and the added edges are between $D$ and $P$ only, (b) every vertex $x\in D$ has degree in $H$ equal to $\kappa(x),$ and (c) every vertex $y \in P$ is of positive degree (since it is so in both $B$ and $H'$). 

\2

\noindent{\bf Case 2: For every vertex $z\in V(G)$, $d_{G}(z)>\kappa(z)$.}  Let $u$ be an arbitrary vertex. Delete any $d_{G}(u)-\kk{u}$ edges incident to $u$ and denote the resulting graph by $L$. Observe that every Nash subgraph of $L$ is a Nash subgraph of $G$ since no vertex in $L$ has degree less than $\kk{u}.$ This reduces Case 2 to Case 1.\qed

\end{document}